\documentclass[1p,12pt]{elsarticle}

\usepackage{amsmath}
\usepackage{amsthm}
\usepackage{mathdots}
\usepackage{amssymb}
\usepackage[ansinew]{inputenc}
\usepackage{url}
\usepackage[all,knot,arc,curve,color,frame]{xy}
\input xy
\xyoption{all}

\newtheorem{lemma}{Lemma}[section]

\newtheorem{theorem}[lemma]{Theorem}
\newtheorem{cor}[lemma]{Corollary}
\newtheorem{prob}[lemma]{Problem}
\newtheorem{prop}[lemma]{Proposition}
\newtheorem{defi}[lemma]{Definition}
\newtheorem{example}[lemma]{Example}
\newtheorem{rem}[lemma]{Remark}
\newtheorem{nota}[lemma]{Notation}

\DeclareMathOperator\dom{dom}
\DeclareMathOperator\ima{im}
\DeclareMathOperator\spa{span}
\DeclareMathOperator\id{id}
\DeclareMathOperator\mo{mod}

\DeclareMathOperator\sac{sac}
\DeclareMathOperator\cs{cs}
\DeclareMathOperator\sym{Sym}

\newcommand{\ra}{\rightarrow}

\newcommand{\miff}{\Leftrightarrow}

\newcommand{\al}{\alpha}
\newcommand{\bt}{\beta}
\newcommand{\gam}{\gamma}
\newcommand{\del}{\delta}
\newcommand{\thet}{\theta}
\newcommand{\vth}{\vartheta}

\newcommand{\lam}{\lambda}
\newcommand{\ome}{\omega}

\newcommand{\vep}{\varepsilon}
\newcommand{\ale}{\aleph}
\newcommand{\aep}{\ale_\vep}
\newcommand{\oep}{\ome_\vep}
\newcommand{\aepp}{\ale_{\vep+1}}
\newcommand{\oepp}{\ome_{\vep+1}}

\newcommand\ara{\stackrel{\al}{\ra}}
\newcommand\arb{\stackrel{\bt}{\ra}}
\newcommand\arga{\stackrel{\gam}{\ra}}
\newcommand\ard{\stackrel{\del}{\ra}}
\newcommand\dax{\downarrow\!\! x}
\newcommand\daz{\downarrow\!\! z}

\newcommand\con{\sim_c}
\newcommand\pp{\mathbb{P}}
\newcommand\mz{\diamond}
\newcommand\lan{\langle}
\newcommand\ran{\rangle}
\newcommand\join{\bigsqcup}
\newcommand\jo{\sqcup}
\newcommand\rro{\mathit{rro}}
\newcommand\cho{\mathit{cho}}

\newcommand\iso{\mathrm{IS}_{\oepp}}
\newcommand\tk{\mbox{{\tiny $K$}}}
\newcommand\tl{\mbox{{\tiny $L$}}}
\newcommand\tq{\mbox{{\tiny $Q$}}}
\newcommand\sa{s(\al)}
\newcommand\sbt{s(\bt)}
\newcommand\ia{i_{\al}}
\newcommand\ja{j_{\al}}

\newcommand\gx{\Gamma(X)}
\newcommand\sqs{\sqsubset}
\newcommand\aal{A_\al}
\newcommand\ba{B_\al}
\newcommand\ca{C_\al}
\newcommand\cna{C^n_\al}
\newcommand\ab{A_\bt}
\newcommand\bb{B_\bt}
\newcommand\cb{C_\bt}
\newcommand\cnb{C^n_\bt}

\newcommand\hph{h_{\phi}}
\newcommand\mi{\mathcal{I}}
\newcommand\sm{\setminus}

\begin{document}

\begin{frontmatter}

\title{Conjugation in Semigroups}

\author[Aberta,CAUL]{Jo\~ao Ara\'ujo\fnref{JAraujo}}
\ead{mjoao@lmc.fc.ul.pt}
\fntext[JAraujo]{Partially supported  by FCT through the following projects: PEst-OE/MAT/UI1043/2011, Strategic Project of Centro de \'Algebra da Universidade de Lisboa; and PTDC/MAT/101993/2008, Project Computations in groups and semigroups.}
\address[Aberta]{Universidade Aberta, R. Escola Polit\'{e}cnica, 147, 1269-001 Lisboa, Portugal}
\address[CAUL]{Centro de \'{A}lgebra da Universidade de Lisboa, 1649-003 Lisboa, Portugal}

\author[MaryWash]{Janusz Konieczny\fnref{Janusz}}
\ead{jkoniecz@umw.edu}
\fntext[Janusz]{Supported by the 2011--12 University of Mary Washington Faculty Research Grant.}
\address[MaryWash]{Department of Mathematics, University of Mary Washington, Fredericksburg, VA 22401.}

\author[Nova,CAUL]{Ant\'onio Malheiro\corref{Malheiro}}
\ead{ajm@fct.unl.pt}
\cortext[Malheiro]{This work was developed within the research activities of Centro de Álgebra da Universidade de Lisboa, FCT´s project PEst-OE/MAT/UI0143/2013, and of Departamento de Matemática da Faculdade de Ciências e Tecnologia da Universidade Nova de Lisboa.}
\address[Nova]{Departamento de Matem\'atica, Faculdade de Ci\^encias e Tecnologia, Universidade Nova de Lisboa, 2829-516 Caparica, Portugal.}

\begin{abstract}
The action of any group on itself by conjugation and
the corresponding conjugacy relation play an important role in group theory.
There have been several attempts to extend the notion of conjugacy to semigroups.
In this paper, we present a new definition of conjugacy that can be applied to an arbitrary semigroup and
it does not reduce to the universal relation in semigroups with a zero.
We compare the new notion of conjugacy with existing definitions,
characterize the conjugacy in various semigroups of transformations on a set,
and count the number of conjugacy classes in these semigroups when the set is infinite.
\end{abstract}


\end{frontmatter}

\section{Introduction}\label{sint}
\setcounter{equation}{0}

Let $G$ be a group. For elements $a,b\in G$, we say that $a$ is \emph{conjugate} to $b$ if there exists $g\in G$ such that $b=g^{-1}ag$.
It is clear that this relation is an equivalence on $G$ and that $a$ is conjugate to $b$ if and only if there
exists $g\in G$ such that $ag=gb$. Using the latter formulation, one may try to extend the
notion of conjugacy to semigroups in the following way: define a relation $\sim_l$ on a semigroup $S$ by
\begin{equation}\label{econ1}
a\sim_l b\Leftrightarrow \exists_{g\in S^1}\ ag=gb,
\end{equation}
where $S^1$ is $S$ with an identity adjoined. If $a\sim_l b$, we say that $a$ is \emph{left conjugate} to $b$ \cite{Ot84,Zh91,Zh92}.
(We will write ``$\sim$'' with various subscripts for
possible definitions of conjugacy in semigroups. The subscript in $\sim_l$ comes from the name ``left conjugate.'')
In a general semigroup $S$, the relation $\sim_l$ is reflexive and transitive, but not symmetric. If $S$ has a zero, then
$\sim_l$ is the universal relation $S\times S$.
The relation $\sim_l$ is an equivalence in any free semigroup. Lallement \cite{La79} has defined the conjugate elements of a free semigroup $S$
as those related by $\sim_l$ and showed that
$\sim_l$ is equal to the following equivalence on the free semigroup $S$:
\begin{equation}\label{econ2}
a\sim_p b\Leftrightarrow \exists_{u,v\in S^1}\ a=uv \mbox{ and } b=vu.
\end{equation}
In a general semigroup $S$, the relation $\sim_p$ is reflexive and symmetric, but not transitive.
If $a\sim_p b$ in a general semigroup, we say that $a$ and $b$ are \emph{primarily conjugate} \cite{KuMa09}
(hence the subscript in $\sim_p$).
Kudryavtseva and  Mazorchuk \cite{KuMa07,KuMa09} have defined
the transitive closure $\sim_p^*$ of $\sim_p$ as
a conjugacy relation in a general semigroup. (See also \cite{Hi06}.)

Otto \cite{Ot84} has studied the relations $\sim_l$ and $\sim_p$ in the monoids $S$ presented by finite Thue systems, and introduced
a new definition of conjugate elements in such an $S$:
\begin{equation}\label{econ3}
a\sim_{\!o} b\Leftrightarrow \exists_{g,h\in S^1}\ ag=gb \mbox{ and } bh=ha.
\end{equation}
(Since $S$ is a monoid, $S^1=S$. However, we wanted to write the definition of $\sim_{\!o}$ so that it would apply to
any semigroup.) For any semigroup $S$, $\sim_{\!o}$ is an equivalence on $S$,
and so it provides another possible definition of conjugacy in a general semigroup.
However, this definition is not useful for semigroups $S$ with zero
since for every such $S$, we have $\sim_{\!o}\,\,=S\times S$. Note that
$\sim_{\!o}$ is the largest equivalence contained in $\sim_l$ and that
$\sim_p\,\,\subseteq\,\,\sim_{\!o}$ since if $a=uv$ and $b=vu$, then $au=ub$ and $bv=va$.

The aim of this paper is to introduce a new definition of conjugacy in an arbitrary
semigroup, avoiding the problems of the notions listed above. (That is, $\sim_l$ is not symmetric;
both $\sim_l$ and $\sim_{\!o}$ reduce to the universal relation
in semigroups with zero; and $\sim_p$ is not transitive and so it requires taking the transitive closure.)
Our conjugacy will be an equivalence relation $\sim_c$ on any semigroup $S$, it will not reduce to the universal relation
even when $S$ has a zero, and it will be
such that $\sim_c\,\,\subseteq\,\,\sim_{\!o}\,\,\subseteq\,\,\sim_l$ in every semigroup $S$, $\sim_c\,\,=\,\,\sim_{\!o}$ if $S$ is a semigroup
without zero,
and $\sim_c\,\,=\,\,\sim_l\,\,=\,\,\sim_p\,\,=\,\,\sim_{\!o}$ if $S$ is a group or a free semigroup.

In Section \ref{scon} we introduce the new notion of conjugacy
and prove some basic results. The following four sections are devoted
to the study of $\sim_c$ in several transformation semigroups on a finite or infinite set $X$.
The tools we use in this study are the characterization
of $\sim_c$ in transformation semigroups in terms of certain
partial homomorphisms of directed graphs (Section~\ref{srph}) and the concept of a connected partial transformation
(Section~\ref{scpt}).
Conjugacy classes in the partial transformation monoid $P(X)$ are characterized (for any $X$)
and counted (for an infinite $X$) in
Section \ref{spx}; conjugacy in the full transformation monoid $T(X)$ is treated in Section \ref{stx};
and Section \ref{sgx} deals with the monoid $\Gamma(X)$ of full injective transformations.
The paper ends with a number of problems for experts in combinatorics, symbolic dynamics, set theory, semigroups, and matrix theory
(Section~\ref{spro}).

\section{Definition of Conjugacy}\label{scon}
\setcounter{equation}{0}

We briefly describe the motivation of our new concept of conjugacy.
The starting point was the relation $\sim_{\!o}$ introduced by Otto \cite{Ot84}.
As we have already pointed out, the relation $\sim_{\!o}$ is the universal relation $S\times S$ if a semigroup
$S$ has a zero. Our goal has been to retain Otto's concept for semigroups without zero but modify his definition
in such a way that the resulting conjugacy would not reduce to triviality for semigroups with zero.

To find a suitable definition, we considered the semigroup $P(X)$ of partial transformations on $X$, that is, the set of
all mappings whose domain and image are included in $X$, with function composition as multiplication. This semigroup has a zero,
namely the transformation whose domain is empty. Let $\al,\bt\in P(X)$. Then $\al\sim_{\!o}\bt$
if and only if $\al\phi=\phi\bt$ and $\bt\psi=\psi\al$ for some $\phi,\psi\in P(X)$. (We write functions on the right
and compose from left to right.) Of course, the last two equalities
hold for $\phi=\psi=0$. We could insist that $\phi$ and $\psi$ should not be $0$ but this would not solve the problem since
then the resulting relation would not be transitive.

The solution is this. In the composition $\al\phi$, it only matters how $\phi$
is defined on the elements of $\ima(\al)$ (the image of $\al$). We insist that $\phi$ be defined for \emph{all}
elements of $\ima(\al)$, that is, that $\ima(\al)\subseteq\dom(\phi)$. With the requirement
that the transformations $\phi$ and $\psi$ come from the sets $\{\phi\in P(X):\ima(\al)\subseteq\dom(\phi)\}$
and $\{\psi\in P(X):\ima(\bt)\subseteq\dom(\psi)\}$, the relation is an equivalence. Moreover,
we will verify that for $\al\ne0$, $\ima(\al)\subseteq\dom(\phi)$ if and only if
$(\gam\al)\phi\ne0$ for every $\gam\al\in P(X)\al\sm \{0\}$, where $P(X)\al\sm \{0\}$ is the left principal ideal
generated by $\al$ with $0$ removed. Therefore, the requirement that $\phi$ and $\psi$ have
``large'' domains can be expressed in an abstract semigroup.
These considerations motivate the definition below.

Let $S$ be a semigroup with zero. For $a\in S$ with $a\ne0$,
consider $S^1a\sm \{0\}$, the left principal ideal generated by $a$ with zero removed. We will denote by
$\pp(a)$ the set of all elements $g\in S$ such that $(ma)g\ne0$ for all $ma\in S^1a\sm \{0\}$. We define
$\pp(0)$ to be $\{0\}$. If $S$ has no zero,
we agree that $\pp(a)=S$ for every $a\in S$. We will write $\pp^1(a)$ for $\pp(a)\cup\{1\}$, where
$1$ is the identity in $S^1$.

\begin{defi}\label{dcon}
{\rm
Define a relation $\con$ on a semigroup $S$ by
\begin{equation}\label{e1dcon}
a\con b\Leftrightarrow \exists_{g\in\pp^1(a)}\exists_{h\in\pp^1(b)}\ ag=gb\textnormal{ and }bh=ha.
\end{equation}
If $a\con b$, we say that $a$ is \emph{conjugate\/} to $b$.
}
\end{defi}

The relation $\sim_c$ will be called \emph{conjugacy} on $S$, which is justified by the following theorem.

\begin{theorem}\label{tcon}
Let $S$ be a semigroup. Then:
\begin{itemize}
  \item [\rm(1)] the relation $\sim_c$ is an equivalence relation on $S$;
  \item [\rm(2)] if $\sim_l$, $\sim_p$, and $\sim_{\!o}$ are relations on $S$ defined by {\rm (\ref{econ1}), (\ref{econ2}), and (\ref{econ3})},
respectively, then:
\begin{itemize}
  \item [\rm(a)] $\sim_c\,\,\subseteq\,\,\sim_{\!o}\,\,\subseteq\,\,\sim_l$,
  \item [\rm(b)] if $S$ is a semigroup without zero, then $\sim_c\,\,=\,\,\sim_{\!o}$, and
  \item [\rm(c)] if $S$ is a group or a free semigroup, then $\sim_c\,\,=\,\,\sim_l\,\,=\,\,\sim_p\,\,=\,\,\sim_{\!o}$.
\end{itemize}
\end{itemize}
\end{theorem}
\begin{proof}
It is clear that $\sim_c$ is reflexive and symmetric. Suppose $a\con b$ and $b\con c$. Then there are $g_1\in\pp^1(a)$
and $g_2\in\pp^1(b)$ such that $ag_1=g_1b$ and $bg_2=g_2c$. Thus $a(g_1g_2)=(ag_1)g_2=(g_1b)g_2=g_1(bg_2)=g_1(g_2c)=(g_1g_2)c$. Let $ma\in S^1a\sm \{0\}$.
Since $g_1\in\pp^1(a)$, we have $(mg_1)b=m(ag_1)=(ma)g_1\ne0$. Thus
$(mg_1)b\in S^1b\sm \{0\}$, and so, since $g_2\in\pp^1(b)$, we have $(ma)(g_1g_2)=m(ag_1)g_2=m(g_1b)g_2=((mg_1)b)g_2\ne0$. Hence $g_1g_2\in\pp^1(a)$.
Similarly, there is $h\in\pp^1(c)$ such that $ch=ha$.
Hence $a\con c$, and so $\sim_c$ is transitive. We have proved (1).

Statements 2(a) and 2(b) follow immediately from the definitions of $\sim_l$, $\sim_{\!o}$, and $\sim_c$.
Statement 2(c) is clearly true if $S$ is a group. Let $S$ be a free semigroup.
Then $\sim_l\,\,=\,\,\sim_p$ by \cite[Corollary~5.2]{La79}. Thus, by 2(a) and 2(b),
$\sim_c\,\,=\,\,\sim_{\!o}\,\,\subseteq\,\,\sim_l\,\,=\,\,\sim_p\,\,\subseteq\,\,\sim_{\!o}$, which implies
$\sim_c\,\,=\,\,\sim_{\!o}\,\,=\,\,\sim_l\,\,=\,\,\sim_p$.
\end{proof}

For an element $a$ of a semigroup $S$, the equivalence class of $a$ with respect to $\sim_c$
will be called the \emph{conjugacy} class of $a$ and denoted $[a]_c$.

Let $S$ be a semigroup with $0$. In contrast with the fact that $\sim_{\!o}\,\,=S\times S$, the conjugacy class of $0$
with respect to $\sim_c$ is $\{0\}$, so we always have $\sim_c\,\,\ne S\times S$ unless $S=\{0\}$.
Indeed, suppose $a\con0$. Then $ag=g0=0$ for some $g\in\pp^1(a)$. If $a\ne0$, then
$ag\ne0$ (since $a\in S^1a\sm \{0\}$). But $ag=0$, and so it follows that $a=0$. Hence we have the following lemma.
\begin{lemma}\label{lccz}
If $S$ is a semigroup with $0$ then $[0]_c=\{0\}$.
\end{lemma}

For a set $A$, we denote by $\Delta_A$ (or $\Delta$ if $A$ is understood) the identity relation on $A$,
that is $\Delta_A=\{(a,a):a\in A\}$. Recall that in any group $G$, the relation $\sim_c$ is the usual group conjugacy,
that is $a\con b$ if and only if $g^{-1}ag=b$ for some $g\in G$. It follows that in any group $G$,
we have $\sim_c\,\,=\Delta$ if and only if $G$ is commutative. This result extends to semigroups as follows.

\begin{theorem}\label{tdel}
Let $S$ be a semigroup without zero. Then $\con\,=\Delta$ if and only if $S$ is commutative and cancellative.
\end{theorem}
\begin{proof}
Since $S$ has no zero, $\con\,=\,\sim_{\!o}$.
It is clear that if $S$ is commutative and cancellative, then $\con\,=\Delta$.
Conversely, suppose that $\con\,=\Delta$.
Let $a,b\in S$. Since $(ab)a=a(ba)$ and $(ba)b=b(ab)$, we have $ab\con ba$, and hence $ab=ba$.
We have proved that $S$ is commutative. Let $a,b,c\in S$ be such that $ac=bc$. Since $S$ is commutative, $ac=bc$ implies
$a\con b$, which in turn implies $a=b$. It follows that $S$ is cancellative.
\end{proof}

Theorem~\ref{tdel} is not true for semigroups with zero. For example, let $S=\{a,0\}$ be a 2-element semigroup
with $aa=0$. Then $S$ is not cancellative but we already know that $[0]_c=\{0\}$,
so $\sim_c\,\,=\Delta$.

\section{Restrictive Partial Homomorphisms of Digraphs}\label{srph}
\setcounter{equation}{0}
\setcounter{figure}{0}

The remainder of the paper is devoted to the study of the conjugacy $\sim_c$ in several important semigroups
of transformations on a set $X$ (finite or infinite). The main tool in our study will be the characterization
of $\sim_c$ in terms of certain partial homomorphisms of directed graphs (see Theorem~\ref{tconrph} and Corollary~\ref{cconh}).

A \emph{directed graph\/} (or a \emph{digraph\/}) is a pair $\Gamma=(X,R)$ where $X$ is a non-empty set (not necessarily finite)
and $R$
is a binary relation on $X$. Any element $x\in X$ is called a \emph{vertex\/} of $\Gamma$,
and any pair $(x,y)\in R$ is called an \emph{arc\/} of $\Gamma$. We will call a vertex $x$ \emph{terminal\/}
if there is no $y\in X$ such that $(x,y)\in R$.

Let $\Gamma_1=(X_1, R_1)$ and $\Gamma_2=(X_2, R_2)$ be digraphs.
A mapping $\phi:X_1\to X_2$
is called a \emph{homomorphism\/} from $\Gamma_1$ to $\Gamma_2$ if for all $x,y\in X_1$,
if $(x,y)\in R_1$, then $(x\phi,y\phi)\in R_2$ \cite{HeNe04}. Generalizing,
a partial mapping $\phi$ from $X_1$ to $X_2$
(that is, a mapping $\phi$ from some subset of $X_1$ to $X_2$)
is called a \emph{partial homomorphism\/} from $\Gamma_1$ to $\Gamma_2$ if for all $x,y\in X_1$,
if $(x,y)\in R_1$ and $x,y\in\dom(\phi)$, then $(x\phi,y\phi)\in R_2$.

\begin{defi}\label{drph}
{\rm
Let $\Gamma_1=(X_1, R_1)$ and $\Gamma_2=(X_2, R_2)$ be digraphs. A partial mapping $\phi$ from $X_1$ to $X_2$
is called a \emph{restrictive partial homomorphism\/} (or an \emph{rp-homomorphism}) from $\Gamma_1$ to $\Gamma_2$ if it satisfies the
following conditions for all $x,y\in X_1$:
\begin{itemize}
  \item [(a)] if $(x,y)\in R_1$, then $x,y\in\dom(\phi)$ and $(x\phi,y\phi)\in R_2$;
  \item [(b)] if $x$ is a terminal vertex in $\Gamma_1$ and $x\in\dom(\phi)$, then $x\phi$ is a terminal vertex in $\Gamma_2$.
\end{itemize}
We say that $\Gamma_1$ is \emph{rp-homomorphic} to $\Gamma_2$ if there is an rp-homomorphism from $\Gamma_1$ to $\Gamma_2$.
}
\end{defi}

Clearly, every rp-homomorphism from $\Gamma_1$ to $\Gamma_2$ is a partial homomorphism from $\Gamma_1$ to $\Gamma_2$.
It is also clear that the composition of rp-homomorphisms is an rp-homomorphism.

\begin{rem}\label{riso}
{\rm
Call a vertex $x$ of a digraph $\Gamma=(X, R)$ \emph{isolated}
if there is no $y\in X$ such that $(x,y)\in R$ or $(y,x)\in R$.
Let $\phi$ be an rp-homomorphism from $\Gamma_1=(X_1, R_1)$ to $\Gamma_2=(X_2, R_2)$.
Denote by $X_1'$ the set of all vertices in $\Gamma_1$ that are not isolated. Then $\phi'=\phi|_{X_1'}$
(the restriction of $\phi$ to $X_1'$)
is also an rp-homomorphism from $\Gamma_1$ to $\Gamma_2$.
}
\end{rem}

In picturing directed graphs, we will adopt the convention that the arrows will be deleted with the understanding
that the arrow goes up along the edge, to the right if the edge is horizontal, and the arrows go counter-clockwise
around a cycle. For example, consider the digraphs $\Gamma_1=(X_1, R_1)$, where $X_1=\{1,2,3,4\}$ and $ R_1=\{(2,3),(3,4)\}$,
and $\Gamma_2=(X_2, R_2)$, where $X_2=\{a,b,c,d\}$ and $ R_2=\{(a,b),(b,d),(c,d)\}$. Then a mapping presented
in Figure~\ref{f41} is a partial homomorphism from $\Gamma_1$ to $\Gamma_2$ (but not a restrictive partial homomorphism),
and a mapping from Figure~\ref{f42} is an rp-homomorphism from $\Gamma_1$ to $\Gamma_2$.

\begin{figure}[h]
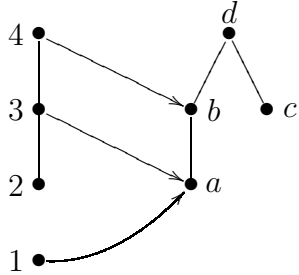

\[
\xy
(0,0)*{\bullet}="1";
(-3,0)*{1};
(0,10)*{\bullet}="2";
(-3,10)*{2};
(0,20)*{\bullet}="3";
(-3,20)*{3};
(0,30)*{\bullet}="4";
(-3,30)*{4};
"4";"2" **\dir{-};
(20,10)*{\bullet}="a";
(23,10)*{a};
(20,20)*{\bullet}="b";
(23,20)*{b};
(25,30)*{\bullet}="d";
(25,33)*{d};
(30,20)*{\bullet}="c";
(33,20)*{c};
"a";"b" **\dir{-};
"d";"b" **\dir{-};
"d";"c" **\dir{-};
"1";"a" **\crv{(10,-1)} ?>* \dir{>};
"3";"a" **\crv{} ?>* \dir{>};
"4";"b" **\crv{} ?>* \dir{>};
\endxy
\]
\caption{A partial homomorphism from $\Gamma_1$ to $\Gamma_2$.}\label{f41}
\end{figure}

\begin{figure}[h]
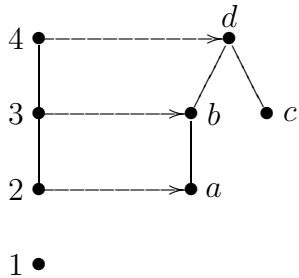

\[
\xy
(0,0)*{\bullet}="1";
(-3,0)*{1};
(0,10)*{\bullet}="2";
(-3,10)*{2};
(0,20)*{\bullet}="3";
(-3,20)*{3};
(0,30)*{\bullet}="4";
(-3,30)*{4};
"4";"2" **\dir{-};
(20,10)*{\bullet}="a";
(23,10)*{a};
(20,20)*{\bullet}="b";
(23,20)*{b};
(25,30)*{\bullet}="d";
(25,33)*{d};
(30,20)*{\bullet}="c";
(33,20)*{c};
"a";"b" **\dir{-};
"d";"b" **\dir{-};
"d";"c" **\dir{-};
"2";"a" **\crv{} ?>* \dir{>};
"3";"b" **\crv{} ?>* \dir{>};
"4";"d" **\crv{} ?>* \dir{>};
\endxy
\]\caption{An rp-homomorphism from $\Gamma_1$ to $\Gamma_2$.}\label{f42}
\end{figure}

Let $\al\in P(X)$. Then $\al$ can be represented by the digraph $\Gamma(\al)=(X,R_\al)$,
where for all $x,y\in X$, $(x,y)\in R_\al$ if and only if $x\in\dom(\al)$ and $x\al=y$
 \cite[Section 1.2]{GaMa10}. If $x\in\dom(\al)$
and $x\al=y$, we will write $x\ara y$ (or $x\ra y$ if no ambiguity arises).
For $\al\in P(X)$, the set $\dom(\al)\cup\ima(\al)$ will be called the \emph{span} of $\al$ and denoted $\spa(\al)$.

For example, the digraph in Figure~\ref{f43} represents the transformation
\[
\al=\begin{pmatrix}
1&2&3&4&5&6&7&8&9&\ldots\\2&3&1&1&1&5&8&9&10&\ldots
\end{pmatrix}\in T(X),
\]
where $X=\{1,2,3,\ldots\}$ and $T(X)$ is the semigroup
of all $\al\in P(X)$ such that $\dom(\al)=X$.

\begin{figure}[h]
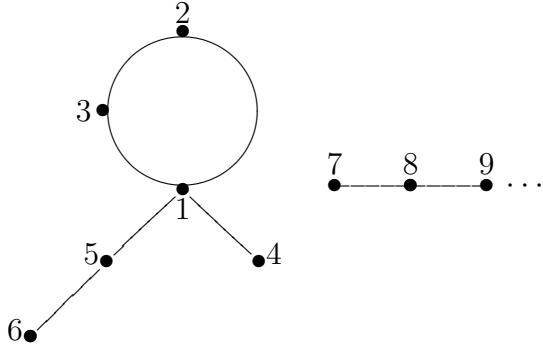

\[
\xy
(50,50)*{}*\cir<30pt>{};
(50,39.5)*{\bullet}="1";
(50,37)*{1};
(50,60.5)*{\bullet}="2";
(50,63)*{2};
(39.5,50)*{\bullet}="3";
(37,50)*{3};
(40,30)*{\bullet}="4";
(38,31)*{5};
(60,30)*{\bullet}="5";
(62,31)*{4};
(30,20)*{\bullet}="6";
(28,21)*{6};
(70,40)*{\bullet}="7";
(70,43)*{7};
(80,40)*{\bullet}="8";
(80,43)*{8};
(90,40)*{\bullet}="9";
(90,43)*{9};
(95,40)*{\cdots};
"4";"1" **\crv{} ?>* \dir{};
"5";"1" **\crv{} ?>* \dir{};
"8";"9" **\crv{} ?>* \dir{};
"7";"8"*{\bullet} **\crv{} ?>* \dir{};
(30,20)*{\bullet};"4" **\crv{} ?>* \dir{};
\endxy
\]
\caption{The digraph of a transformation.}\label{f43}
\end{figure}

\begin{defi}\label{dcr}
{\rm
Any $\al\in P(X)$ with $\ima(\al)=\{x\}$ will be called a \emph{constant}.
A subsemigroup $S$ of $P(X)$ will be called \emph{constant rich} if for every $x\in X$,
there is $\al\in S$ such that $\ima(\al)=\{x\}$.
}
\end{defi}

Among the constant rich subsemigroups of $P(X)$, we have $P(X)$ itself (an all its nonzero ideals),
the full transformation semigroup $T(X)$ (and all its ideals), and the
symmetric inverse semigroup $\mi(X)$ of all injective $\al\in P(X)$
(and all its nonzero ideals).

\begin{nota}\label{nzero}
{\rm
From now on, we will fix a nonempty set $X$ and an element $\mz\notin X$. For $\al\in P(X)$ and $x\in X$, we will write $x\al=\mz$
if and only if $x\notin\dom(\al)$. We will also assume that $\mz\al=\mz$. With this notation, it will make sense
to write $x\al=y\bt$ or $x\al\ne y\bt$ ($\al,\bt\in P(X)$, $x,y\in X$) even when $x\notin\dom(\al)$ or $y\notin\dom(\bt)$.

We will also denote by $\mathbb Z$, $\mathbb Z_+$, and $\mathbb N$ the set of integers, positive integers, and nonnegative
integers, respectively, and for semigroups $S$ and $T$, write $S\leq T$ to mean that $S$ is a subsemigroup of $T$.
}
\end{nota}

\begin{lemma}\label{lgpa}
Let $S\leq P(X)$ such that $S$ is constant rich, let $\al\in S$ with $\al\ne0$, and $\phi\in S^1$. Then:
\begin{itemize}
  \item [\rm(1)] $\phi\in\pp^1(\al)$ if and only if $\ima(\al)\subseteq\dom(\phi)$;
  \item [\rm(2)] if $\phi\in\pp^1(\al)$ and $\al\phi=\phi\bt$ for some $\bt\in S$, then $\spa(\al)\subseteq\dom(\phi)$
and for all $x,y\in X$, $x\ara y$ implies $x\phi\arb y\phi$.
\end{itemize}
\end{lemma}
\begin{proof}
Let $S$ be constant rich.
Suppose $\phi\in\pp^1(\al)$. Let $y\in\ima(\al)$, that is, $y=x\al$ for some $x\in\dom(\al)$.
Since $S$ is constant rich, there is $\gam\in S$ with $\ima(\gam)=\{x\}$. Then
$\ima(\gam\al)=\{y\}$, and so $\gam\al\in S^1\al\sm \{0\}$. Thus $(\gam\al)\phi\ne0$ (since $\phi\in\pp^1(\al)$),
which is only possible when $y\in\dom(\phi)$. Hence $\ima(\al)\subseteq\dom(\phi)$.

Conversely, suppose $\ima(\al)\subseteq\dom(\phi)$. Let $\mu\al\in S^1\al\sm \{0\}$. Since $\mu\al\ne0$, there is $x\in X$
such that $x(\mu\al)\ne0$. But then $x(\mu\al)=(x\mu)\al\in\ima(\al)\subseteq\dom(\phi)$, and so $x\in\dom((\mu\al)\phi)$. Thus
$(\mu\al)\phi\ne0$, and so $\phi\in\pp^1(\al)$. We have proved (1).

To prove (2), suppose $\phi\in\pp^1(\al)$ and $\al\phi=\phi\bt$ for some $\bt\in S$.
Let $x,y\in X$ and suppose that $x\ara y$. Then, since $\al\phi=\phi\bt$, we have
\begin{equation}
(x\phi)\bt=x(\phi\bt)=x(\al\phi)=(x\al)\phi=y\phi.\label{e1lgpa}
\end{equation}
By (1), $\ima(\al)\subseteq\dom(\phi)$, and so $y=x\al\in\dom(\phi)$.
Then, by (\ref{e1lgpa}), $x\phi\ne\mz$, which implies $x\in\dom(\phi)$.
It follows that $\spa(\al)\subseteq\dom(\phi)$. Moreover, by (\ref{e1lgpa}) again,
$(x\phi)\bt=y\phi\ne\mz$, and so $x\phi\arb y\phi$.
\end{proof}

\begin{lemma}\label{lspa}
Let $\al,\bt\in P(X)$ and let $\phi$ be an rp-homomorphism from $\Gamma(\al)$ to $\Gamma(\bt)$. Then
$\spa(\al)\subseteq\dom(\phi)$.
\end{lemma}
\begin{proof}
Let $x\in\spa(\al)$. If $x\in\dom(\al)$, then $x\ara x\al$, and so $x,x\al\in\dom(\phi)$ by Definition~\ref{drph}.
If $x\in\ima(\al)$, then $z\ara x$ for some $z\in\dom(\al)$, and so $z,x\in\dom(\phi)$. Hence $\spa(\al)\subseteq\dom(\phi)$.
\end{proof}

\begin{lemma}\label{laggb}
Let $S\leq P(X)$ such that $S$ is constant rich, let $\al,\bt\in S$ with $\al\ne0$, and $\phi\in S^1$.
Then $\al\phi=\phi\bt$ with $\phi\in\pp^1(\al)$
if and only if $\phi$ is an rp-homomorphism from $\Gamma(\al)$ to $\Gamma(\bt)$.
\end{lemma}
\begin{proof}
Suppose $\al\phi=\phi\bt$ with $\phi\in\pp^1(\al)$.
Let $x,y\in X$ and suppose that $x\ara y$. Then $x\phi\arb y\phi$ by Lemma~\ref{lgpa},
and so $\phi$ satisfies (a) of Definition~\ref{drph}.
Suppose that $x$ is a terminal vertex of $\Gamma(\al)$ and $x\in\dom(\phi)$. Then $x\phi\in X$ and $x\al=\mz$.
Since $\al\phi=\phi\bt$, we have $(x\phi)\bt=(x\al)\phi=\mz\phi=\mz$, and so $x\phi$ is a terminal vertex in $\Gamma(\bt)$.
Hence $\phi$ satisfies (b) of Definition~\ref{drph}. Thus $\phi$ is an rp-homomorphism from $\Gamma(\al)$ to $\Gamma(\bt)$.

Conversely, suppose that $\phi$ is an rp-homomorphism from $\Gamma(\al)$ to $\Gamma(\bt)$. Let $x\in X$.
Suppose $x\notin\dom(\al)$. Then $x(\al\phi)=(x\al)\phi=\mz\phi=\mz$. If $x\notin\dom(\phi)$, then $x(\phi\bt)=(x\phi)\bt=\mz\bt=\mz$.
If $x\in\dom(\phi)$, then, by (b) of Definition~\ref{drph}, $x\phi$ is a terminal vertex in $\Gamma(\bt)$, and so
$x(\phi\bt)=(x\phi)\bt=\mz$. Hence, in both cases, $x(\al\phi)=x(\phi\bt)$.

Suppose $x\in\dom(\al)$ and let $y=x\al\in X$. Then $x\ara y$, and so, by Definition~\ref{drph},
$x,y\in\dom(\phi)$ and $x\phi\arb y\phi$. Hence
$x(\al\phi)=(x\al)\phi=y\phi$ and $x(\phi\bt)=(x\phi)\bt=y\phi$, and so $x(\al\phi)=x(\phi\bt)$.
We have proved that $\al\phi=\phi\bt$. Finally, since $\phi$ is an rp-homomorphism from $\Gamma(\al)$ to $\Gamma(\bt)$,
we have that $\spa(\al)\subseteq\dom(\phi)$ by Lemma~\ref{lspa}, and so $\phi\in\pp^1(\al)$ by Lemma~\ref{lgpa}.
\end{proof}

\begin{theorem}\label{tconrph}
Let $S\leq P(X)$ such that $S$ is constant rich, let $\al,\bt\in S$.
Then $\al\con\bt$ in $S$ if and only if there are $\phi,\psi\in S^1$
such that $\phi$ is an rp-homomorphism from $\Gamma(\al)$ to $\Gamma(\bt)$
and $\psi$ is an rp-homomorphism from $\Gamma(\bt)$ to $\Gamma(\al)$.
\end{theorem}
\begin{proof}
Suppose $\al\con\bt$.
If $\al=0$, then $\bt=0$ (since $[0]_c=\{0\}$), and so $\phi=\id_X\in S^1$ is an rp-homomorphism
from $\Gamma(\al)$ to $\Gamma(\bt)$. Suppose $\al\ne0$. Since $\al\con\bt$, there is $\phi\in\pp^1(\al)$
such that $\al\phi=\phi\bt$, and so
$\phi$ is an rp-homomorphism from $\Gamma(\al)$ to $\Gamma(\bt)$ by Lemma~\ref{laggb}.
A desired $\psi\in S^1$ exists by symmetry.

Conversely, suppose that desired $\phi$ and $\psi$ exist. If $x\ara y$ then $x\phi\arb y\phi$,
and if $x\arb y$ then $x\psi\ara y\psi$. It follows that either $\al=\bt=0$ or $\al,\bt\ne0$. In the former
case, we clearly have $\al\con\bt$. Suppose $\al,\bt\ne0$.
Then, by Lemma~\ref{laggb},
$\al\phi=\phi\bt$ with $\phi\in\pp^1(\al)$ and $\bt\psi=\psi\al$ with $\psi\in\pp^1(\bt)$, which implies
$\al\con\bt$.
\end{proof}

Let $\al,\bt\in T(X)$. Then the graph $\Gamma(\al)$ has no terminal vertices
(if $x\in X$, then $x\ara x\al$), and so every homomorphism from $\Gamma(\al)$ to $\Gamma(\bt)$ is restrictive. This observation
and Theorem~\ref{tconrph} give us the following corollary.

\begin{cor}\label{cconh}
Let $S\leq T(X)$ such that $S$ contains all constants, and let $\al,\bt\in S$. Then $\al\con\bt$ in $S$ if and only if there are $\phi,\psi\in S^1$
such that $\phi$ is a homomorphism from $\Gamma(\al)$ to $\Gamma(\bt)$ and
$\psi$ is a homomorphism from $\Gamma(\bt)$ to $\Gamma(\al)$.
\end{cor}

\section{Connected Partial Transformations}\label{scpt}
\setcounter{equation}{0}
\setcounter{figure}{0}

In this section, we introduce the concept of connected partial transformation. The definitions and results of this section
will be crucial in characterizing conjugacy in various semigroups of transformations.

\begin{defi}\label{dbas}
{\rm
Let $\ldots,x_{-2},x_{-1},x_0,x_1,x_2,\ldots$ be pairwise distinct elements of $X$. The following elements of $P(X)$
will be called \emph{basic} partial transformations on $X$.
\begin{itemize}
  \item [(1)] A \emph{cycle} of length $k$ ($k\geq1$), written $(x_0\,x_1\ldots\, x_{k-1})$,
is an element of $P(X)$ defined by the digraph $x_0\ra x_1\ra \cdots\ra x_{k-1}\ra x_0$.
  \item [(2)] A \emph{right ray}, written $[x_0\,x_1\,x_2\ldots\ran$,
is an element of $P(X)$ defined by the digraph \[x_0\ra x_1\ra x_2\ra \cdots.\]
  \item [(3)] A \emph{double ray}, written $\lan\ldots\,x_{-1}\,x_0\,x_1\ldots\ran$,
is an element of $P(X)$ defined by the digraph \[\cdots\ra x_{-1}\ra x_0\ra x_1\ra\cdots.\]
  \item [(4)] A \emph{left ray}, written $\lan\ldots\, x_2\,x_1\,x_0]$,
is an element of $P(X)$ defined by the digraph \[\cdots \ra x_2\ra x_1\ra x_0.\]
  \item [(5)] A \emph{chain} of length $k$ ($k\geq1$), written $[x_0\,x_1\ldots\, x_k]$,
is an element of $P(X)$ defined by the digraph $x_0\ra x_1\ra \cdots\ra x_k$.
\end{itemize}
By a \emph{ray} we will mean a double, right, or left ray.
}
\end{defi}

We note the following.
\begin{itemize}
  \item [(i)] All basic partial transformations are injective.
  \item [(ii)] The span of a basic partial transformation is exhibited by the notation. For example, the span
of the right ray $[1\,2\,3\ldots\ran$ is $\{1,2,3,\ldots\}$.
  \item [(iii)] The left bracket in ``$\sigma=[x\ldots$'' indicates that $x\notin\ima(\sigma)$; while the right bracket
in ``$\sigma=\ldots\,x]$'' indicates that $x\notin\dom(\sigma)$. For example, for the chain $\sigma=[1\,2\,3\,4]$,
$\dom(\sigma)=\{1,2,3\}$ and $\ima(\sigma)=\{2,3,4\}$.
  \item [(iv)] A cycle $(x_0\,x_1\ldots\, x_{k-1})$ differs from the corresponding cycle in the symmetric group
of permutations on $X$ in that the former is undefined for every $x\in X\sm \{x_0,x_1,\ldots,x_{k-1}\}$ while the latter
is defined on and fixes every such $x$.
\end{itemize}

\begin{defi}\label{dcon1}
{\rm
An element $\gamma\in P(X)$ is called \emph{connected} if $\gamma\ne0$ and for all $x,y\in\spa(\gamma)$,
$x\gamma^k=y\gamma^m\ne\mz$ for some integers $k,m\geq0$ (where $\gamma^0=\id_X$).
}
\end{defi}

We note that a nonzero $\gamma\in P(X)$ is connected if and only if the underlying undirected graph of the
digraph $\Gamma^0(\gamma)$ is connected,
where $\Gamma^0(\gamma)$ is the digraph $\Gamma(\gamma)$
with the isolated vertices removed, and that all basic partial transformations are connected.

\begin{defi}\label{ddis}
{\rm
Let $\alpha,\beta\in P(X)$. We say that $\beta$ is \emph{contained} in $\alpha$ (or $\al$ \emph{contains} or \emph{has} $\bt$),
and write $\bt\sqs\al$,
if $\dom(\bt)\subseteq\dom(\al)$ and $x\bt=x\al$ for every $x\in\dom(\bt)$. In other words,
$\bt\sqs\al$ iff $\bt=\vep\al$ where $\vep$ is the identity on the domain of $\bt$.
We say that $\al$ and $\bt$ are \emph{disjoint} if $\dom(\al)\cap\dom(\bt)=\emptyset$;
they are \emph{completely disjoint} if $\spa(\al)\cap\spa(\bt)=\emptyset$.
}
\end{defi}

For example, the right ray $[3\,4\,5\,6\ldots\ran$ and chain $[0\,1\,2\,5]$ in $P(\mathbb Z)$
are disjoint
but not completely disjoint. Their join $[3\,4\,5\,6\ldots\ran\jo[0\,1\,2\,5]$ (see Definition~\ref{djoi} below)
is connected.

\begin{defi}\label{djoi}
{\rm
Let $C$ be a set of pairwise disjoint elements of $P(X)$. The \emph{join} of the elements of $C$,
denoted $\join_{\gam\in C}\gam$, is an element of $P(X)$ defined by
\[
x\left(\join_{\gam\in C}\gam\right)=
\left\{\begin{array}{ll}
x\gam & \mbox{if $x\in\dom(\gam)$ for some $\gam\in C$,}\\
\mz & \mbox{otherwise.}
\end{array}\right.
\]
If $C=\{\gam_1,\gam_2,\ldots,\gam_k\}$ is finite, we may write $\join_{\gam\in C}\gam$ as $\gam_1\jo\gam_2\jo\cdots\jo\gam_k$.
}
\end{defi}

\begin{prop}\label{pdec}
Let $\al\in P(X)$ with $\al\ne0$. Then there exists a unique set $C$ of pairwise completely disjoint, connected transformations
contained in $\al$ such that $\al=\join_{\gam\in C}\gam$.
\end{prop}
\begin{proof}
Define a relation $ R$ on $\dom(\al)$ by: $(x,y)\in R$ if $x\al^k=y\al^m\ne\mz$ for some integers $k,m\geq0$.
It is clear that $ R$ is an equivalence relation on $\dom(\al)$. Let $J$ be a complete set of representatives of the
equivalence classes of $ R$. For every $x\in J$, let $\gam_x=\al|_{x R}$, where $xR$ is the $R$-equivalence
class of $x$. By the definition of $ R$, each such $\gam_x$ is connected, and $\gam_x$ and $\gam_y$ are completely disjoint
for all $x,y\in J$ with $x\ne y$. Then the set $C=\{\gam_x:x\in J\}$ consists of pairwise completely disjoint, connected
transformations contained in $\al$, and $\al=\join_{\gam\in C}\gam$.

Suppose $D$ is any set of pairwise completely disjoint, connected transformations contained in $\al$ such that
$\al=\join_{\del\in D}\del$. Let $\del\in D$ and let $y\in\dom(\del)$. Then $y\in x R$ for some $x\in J$.
We want to prove that $\del=\gam_x$. Let $z\in\dom(\del)$. Since $\del$ is connected, $y\del^k=z\del^m\ne\mz$
for some $k,m\geq0$. But then, since $\del$ is contained in $\al$, we have $y\al^k=z\al^m\ne\mz$. Hence
$(y,z)\in R$, and so $z\in yR=xR=\dom(\gam_x)$. We have proved that $\dom(\del)\subseteq\dom(\gam_x)$.

Suppose to the contrary that $\dom(\gam_x)$ is not included in $\dom(\del)$, that is, that there is $w\in\dom(\gam_x)$
such that $w\notin\dom(\del)$. Since $\gam_x$ is connected, $y\gam_x^p=w\gam_x^q\ne\mz$ for some $p,q\geq0$.
Let $y_i=y\gam_x^i=y\al^i$ and $w_j=w\gam_x^j=w\al^j$ for $i=0,1,\ldots,p$ and $j=0,1,\ldots,q$. Then $y_p=w_q$ and
let $u=y_p=w_q$. With this notation, in the digraph $\Gamma(\al)$, we have
\[
y=y_0\ra y_1\ra\cdots\ra y_p=u\textnormal{ and }w=w_0\ra w_1\ra\cdots\ra w_q=u.
\]
We claim that $\{y_0,y_1,\ldots,y_{p-1}\}\subseteq\dom(\del)$. If not, then, since $y_0=y\in\dom(\delta)$,
there would be $i\in\{0,\ldots,p-2\}$ such that $y_i\in\dom(\del)$ and $y_{i+1}\notin\dom(\del)$. But
$y_{i+1}\in\dom(\al)$, and so $y_{i+1}\in\dom(\del_1)$ for some $\del_1\in D$. We would then have $\del\ne\del_1$
and $y_{i+1}\in\spa(\del)\cap\spa(\del_1)$, which is impossible since $\del$ and $\del_1$ are completely disjoint.
The claim has been proved.
Since $w\in\dom(\gam_x)\subseteq\dom(\al)$, there is $\delta_2\in D$ such that $w\in\dom(\delta_2)$.
By the foregoing argument applied to $\del_2$ and $\{w_0,w_1,\ldots,w_{q-1}\}$, we obtain
$\{w_0,w_1,\ldots,w_{q-1}\}\subseteq\dom(\del_2)$. Thus
\[
y_{p-1}\del=y_{p-1}\al=y_p=u=w_q=w_{q-1}\al=w_{q-1}\del_2.
\]
Thus we have $\del\ne\del_2$ with $u\in\ima(\del)\cap\ima(\del_2)$, which is a contradiction since
$\del$ and $\del_2$ are completely disjoint. We have proved that $\dom(\gam_x)\subseteq\dom(\del)$, and so
$\dom(\del)=\dom(\gam_x)$. Now for all $v\in\dom(\del)=\dom(\gam_x)$, we have
$v\del=v\al=v\gam_x$, and so $\del=\gam_x\in C$. We have proved that $D\subseteq C$.

For the reverse inclusion, let $\gam_x$ be an arbitrary element of $C$. Select $y\in\dom(\gam_x)$. Then,
there is $\del\in D$ such that $y\in\dom(\del)$. By the foregoing argument, we have $\del=\gam_x$, and so
$\gam_x\in D$. Hence $C\subseteq D$, and so $D=C$. We have proved that the set $C$ is unique, which completes the proof.
\end{proof}

Any element of the set $C$ from Proposition~\ref{pdec} will be called a \emph{connected component} of $\al$.
We note that the connected components of $\al$ correspond to the connected components of the underlying
undirected graph of $\Gamma(\al)$ that are not isolated vertices.

\begin{defi}\label{dmax}
{\rm
Let $\al\in P(X)$ and let $\mu$ be a basic partial transformation contained in $\al$.
We say that $\mu$ is \emph{maximal} in $\al$ if for every $x\in\spa(\mu)$, $x\notin\dom(\mu)$ implies $x\notin\dom(\al)$,
and $x\notin\ima(\mu)$ implies $x\notin\ima(\al)$. Note that if $\mu$ is a cycle or a double ray, then
$\mu$ is always maximal in $\al$.
}
\end{defi}

For example, consider $\al=[3\,4\,5\,6\ldots\ran\jo[0\,1\,2\,5]\in P(\mathbb Z)$. Then $\al$ contains infinitely many right rays,
for example $[2\,5\,6\,7\ldots\ran$, but only two of them, namely $[3\,4\,5\,6\ldots\ran$ and $[0\,1\,2\,5\,6\,7\ldots\ran$
are maximal. Also, $\al$ contains infinitely many chains, for example $[3\,4\,5\,6]$, but none of them is maximal.

We will now establish which combinations of basic partial transformations can occur
in a connected element of $P(X)$.

\begin{lemma}\label{lcom}
Let $\gam\in P(X)$ be connected.
\begin{itemize}
  \item [\rm(1)] If $\gam$ has a cycle $(x_0\,x_1\ldots\,x_{k-1})$, then for every $x\in\dom(\gam)$,
$x\gam^m=x_0$ for some $m\geq0$.
  \item [\rm(2)] If $\gam$ has a right ray $[x_0\,x_1\,x_2\ldots\,\ran$ or a double ray
$\lan\ldots\,x_{-1}\,x_0\,x_1\ldots\ran$, then for every $x\in\dom(\gam)$,
$x\gam^m=x_i$ for some $m,i\geq0$.
  \item [\rm(3)] If $\gam$ has a maximal chain $[x_k\ldots\,x_1\,x_0]$ or
a maximal left ray $\lan\ldots\,x_2\,x_1\,x_0]$, then for every $x\in\spa(\gam)$,
$x\gam^m=x_0$ for some $m\geq0$.
\end{itemize}
\end{lemma}
\begin{proof}
Suppose $\gam$ has a cycle $(x_0\,x_1\ldots\,x_{k-1})$ and let $x\in\dom(\gam)$. Since $\gam$ is connected,
$x\gam^p=x_0\gam^q$ for some $p,q\geq0$. Since $x_0$ lies on the cycle $(x_0\,x_1\ldots\,x_{k-1})$,
we may assume that $0\leq q\leq k-1$. Thus for $m=p+k-q$, we have
\[
x\gam^m=x\gam^{p+k-q}=(x\gam^p)\gam^{k-q}=(x_0\gam^q)\gam^{k-q}=x_q\gam^{k-q}=x_0.
\]

Suppose $\gam$ has a right ray $[x_0\,x_1\,x_2\ldots\ran$ and let $x\in\dom(\gam)$. Since $\gam$ is connected,
$x\gam^m=x_0\gam^i=x_i$ for some $m,i\geq0$. A proof in the case of a double ray is the same.

Suppose $\gam$ has a chain $[x_k\ldots\,x_1\,x_0]$ and let $x\in\spa(\gam)$. Since $\gam$ is connected,
$x\gam^p=x_0\gam^q\ne\mz$ for some $p,q\geq0$. Note that $q$ must be $0$ (since $x_0\gam^q=\mz$ for every $q\geq1$).
Thus $x\gam^p=x_0\gam^0=x_0$. The proof in the case of a maximal left ray is the same.
\end{proof}

\begin{prop}\label{pcom}
Let $\gam\in P(X)$ be connected.
\begin{itemize}
  \item [\rm(1)] If $\gam$ has a cycle, then the cycle is unique and $\gam$ does not have any double rays or right rays
or maximal chains or maximal left rays.
  \item [\rm(2)] If $\gam$ has a double ray, then it does not have any maximal chains or maximal left rays.
  \item [\rm(3)] If $\gam$ has a right ray, but no double rays, then it has a maximal right ray
and it does not have any left rays or maximal chains.
  \item [\rm(4)] If $\gam$ has a chain, but no cycles or rays, then it has a maximal chain.
  \item [\rm(5)] If $\gam$ has a left ray, but no cycles or double rays, then it has a maximal left ray.
\end{itemize}
\end{prop}
\begin{proof}
Suppose that $\gam$ has a cycle, and let $\thet$ and $\vth$ be cycles in $\gam$,
say $\thet=(x_0\,x_1\ldots\,x_{k-1})$. Let $y\in\dom(\vth)$.
By Lemma~\ref{lcom}, $y\gam^p=x_0$ for some $p\geq0$. Thus $x_0$ lies on $\vth$, so we may write
$\vth=(y_0\,y_1\ldots\,y_{m-1})$ with $y_0=x_0$. We may assume that $k\leq m$.
But then $x_i=x_0\gam^i=y_0\gam^i=y_i$ for every $i\in\{0,\ldots,k-1\}$ and $y_{k-1}\gam=x_{k-1}\gam=x_0=y_0$,
that is, $\thet=\vth$.

Suppose that $\gam$ with a cycle $(x_0\,x_1\ldots\,x_{k-1})$ also has a double ray, say $\lan\ldots\,y_{-1}\,y_0\,y_1\ldots\ran$.
By Lemma~\ref{lcom}, $y_0\gam^m=x_0$ for some $m\geq0$. But then $y_0\gam^{m+k}=(y_0\gam^m)\gam^k=x_0\gam^k=x_0=y_0$,
which is a contradiction since $y_0\gam^{m+k}=y_{m+k}\ne y_0$ (since $m\geq0$ and $k\geq1$). Thus $\gam$ does not have a double ray.
This completes the proof of (1) since a connected $\gam$ with a cycle cannot have any terminal vertices, and hence
cannot have any maximal chains or maximal left rays. Statement (2) also follows since a connected $\gam$ with a double ray
cannot have terminal vertices either.

Let $\eta=[x_0\,x_1\,x_2\ldots\ran$ be a right ray in $\al$.
If $\eta$ is not maximal, then
$x_{-1}\gam=x_0$ for some $x_{-1}\in X\sm \{x_0,x_1,\ldots\}$. (If $x_{-1}=x_i$ for some $i\geq0$, then $\gam$ would have a cycle,
which is impossible by (1).)
Thus $\eta_1=[x_{-1}\,x_0\,x_1\,x_2\ldots\ran$ is a right ray in $\al$.
If $\eta_1$ is not maximal, then
$x_{-2}\gam=x_{-1}$ for some $x_{-2}\in X\sm \{x_{-1},x_0,x_1,\ldots\}$, and so
$\eta_2=[x_{-2}\,x_{-1}\,x_0\,x_1\,x_2\ldots\ran$ is a right ray in $\al$.
Continuing this way, we must arrive at a maximal right ray in $\al$ (after finitely many steps) since otherwise
$\al$ would have a double ray.
This completes the proof of (3) since a connected $\gam$ with a right ray cannot have any terminal vertices.

To prove (4), let $\lam=[x_0\,x_1\ldots\,x_k]$ be a chain in $\al$.
If $x_0\in\ima(\al)$, then, since $\al$ has no left rays,
we can use the argument as in the proof of (3) for a right ray to extend $\lam$
to a chain $\lam'=[x_{-m}\ldots x_{-1}\,x_0\,x_1\ldots\,x_k]$ such that $x_{-m}\notin\ima(\al)$.
Similarly, since $\al$ has no right rays or cycles, we can extend $\lam'$ to a chain
$\lam''=[x_{-m}\,x_{-m+1}\ldots x_{-1}\,x_0\,x_1\ldots\,x_k\,x_{k+1}\ldots\,x_{k+p}]$
such that $x_{k+p}\notin\dom(\al)$. Then $\lam''$ is a maximal chain in $\al$. We have proved~(4).
The proof of (5) is similar.
\end{proof}

\begin{rem}\label{rtyp}
{\rm
It follows from Proposition~\ref{pcom} that as far as the types of basic transformations go,
a connected $\gam\in P(X)$ can contain one of the following.
\begin{itemize}
  \item [(1)] A single cycle and no double rays or right rays or maximal chains or maximal left rays (see Figure~\ref{f51});
  \item [(2)] A double ray but no cycles or maximal chains or maximal left rays (see Figure~\ref{f52});
  \item [(3)] A maximal right ray but no cycles or double rays or left rays or maximal chains (see Figure~\ref{f53});
  \item [(4)] A maximal left ray but no cycles or double rays or right rays (see Figure~\ref{f54} and Definition~\ref{drro});
  \item [(5)] A maximal chain but no cycles or rays (see Figure~\ref{f55} and Definition~\ref{drro}).
\end{itemize}
We note that the uniqueness
applies only to a cycle. A connected $\gam$ can have any number (finite or infinite) of (maximal) chains or (maximal) rays of any type.

}
\end{rem}

\begin{figure}[h]
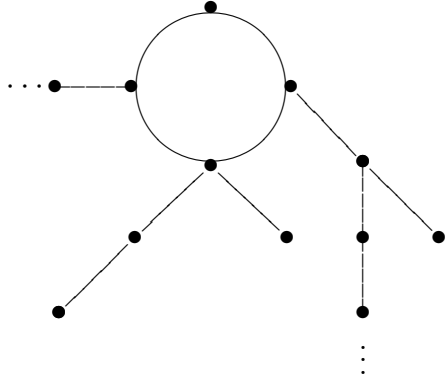

\[
\xy
(50,50)*{}*\cir<30pt>{};
(50,39.5)*{\bullet}="1";
(50,60.5)*{\bullet}="2";
(39.5,50)*{\bullet}="3";
(60.5,50)*{\bullet}="3b";
(29.5,50)*{\bullet}="3c";
(25.5,50)*{\cdots};
(40,30)*{\bullet}="4";
(60,30)*{\bullet}="5";
(30,20)*{\bullet}="6";
(70,40)*{\bullet}="7";
(80,30)*{\bullet}="7b";
(70,30)*{\bullet}="8";
(70,20)*{\bullet}="9";
(70,15)*{\vdots};
"7";"3b" **\crv{} ?>* \dir{};
"3c";"3" **\crv{} ?>* \dir{};
"7b";"7" **\crv{} ?>* \dir{};
"4";"1" **\crv{} ?>* \dir{};
"5";"1" **\crv{} ?>* \dir{};
"9";"8" **\crv{} ?>* \dir{};
"8";"7"*{\bullet} **\crv{} ?>* \dir{};
(30,20)*{\bullet};"4" **\crv{} ?>* \dir{};
\endxy
\]
\caption{A connected partial transformation with a cycle.}\label{f51}
\end{figure}

\begin{figure}[h]
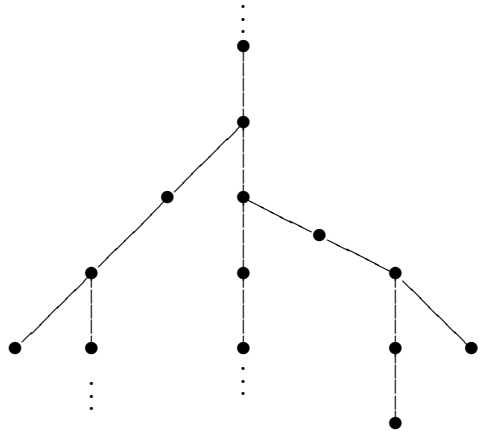

\[
\xy
(80,75)*{\vdots};
(80,27)*{\vdots};
(60,30)*{\bullet}="3";
(60,25)*{\vdots};
(50,30)*{\bullet}="4";
(60,40)*{\bullet}="5";
(70,50)*{\bullet}="6";
(80,30)*{\bullet}="7a";
(80,40)*{\bullet}="7";
(80,50)*{\bullet}="8";
(80,60)*{\bullet}="9";
(80,70)*{\bullet}="9a";
(100,30)*{\bullet}="13";
(100,20)*{\bullet}="13b";
(90,45)*{\bullet}="14";
(100,40)*{\bullet}="15";
(110,30)*{\bullet}="16";
"7a";"7" **\crv{} ?>* \dir{-};
"9a";"9" **\crv{} ?>* \dir{-};
"3";"5" **\crv{} ?>* \dir{-};
"4";"5" **\crv{} ?>* \dir{-};
"5";"6" **\crv{} ?>* \dir{-};
"6";"9" **\crv{} ?>* \dir{-};
"7";"8" **\crv{} ?>* \dir{-};
"8";"9" **\crv{} ?>* \dir{-};
"13";"15" **\crv{} ?>* \dir{-};
"14";"15" **\crv{} ?>* \dir{-};
"15";"16" **\crv{} ?>* \dir{-};
"13b";"13" **\crv{} ?>* \dir{-};
"14";"8" **\crv{} ?>* \dir{-};
\endxy
\]
\caption{A connected partial transformation with a double ray.}\label{f52}
\end{figure}

\begin{figure}[h]
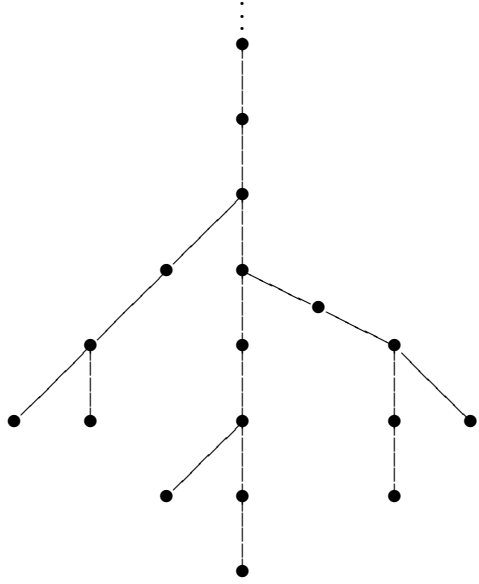

\[
\xy
(80,85)*{\vdots};
(60,30)*{\bullet}="3";
(50,30)*{\bullet}="4";
(60,40)*{\bullet}="5";
(70,50)*{\bullet}="6";
(80,40)*{\bullet}="7";
(80,30)*{\bullet}="7b";
(80,20)*{\bullet}="7c";
(70,20)*{\bullet}="7cc";
(80,10)*{\bullet}="7d";
(80,50)*{\bullet}="8";
(80,60)*{\bullet}="9";
(80,70)*{\bullet}="9a";
(80,80)*{\bullet}="9b";
(100,30)*{\bullet}="13";
(100,20)*{\bullet}="13b";
(90,45)*{\bullet}="14";
(100,40)*{\bullet}="15";
(110,30)*{\bullet}="16";
"9";"9a" **\crv{} ?>* \dir{-};
"9a";"9b" **\crv{} ?>* \dir{-};
"3";"5" **\crv{} ?>* \dir{-};
"4";"5" **\crv{} ?>* \dir{-};
"5";"6" **\crv{} ?>* \dir{-};
"7d";"7c" **\crv{} ?>* \dir{-};
"7c";"7b" **\crv{} ?>* \dir{-};
"7cc";"7b" **\crv{} ?>* \dir{-};
"7b";"7" **\crv{} ?>* \dir{-};
"6";"9" **\crv{} ?>* \dir{-};
"7";"8" **\crv{} ?>* \dir{-};
"8";"9" **\crv{} ?>* \dir{-};
"13";"15" **\crv{} ?>* \dir{-};
"14";"15" **\crv{} ?>* \dir{-};
"15";"16" **\crv{} ?>* \dir{-};
"13b";"13" **\crv{} ?>* \dir{-};
"14";"8" **\crv{} ?>* \dir{-};
\endxy
\]
\caption{A connected partial transformation of type $\rro$.}\label{f53}
\end{figure}

\begin{figure}[h]
\[
\xy
(80,17)*{\vdots};
(60,30)*{\bullet}="3";
(60,27)*{\vdots};
(50,30)*{\bullet}="4";
(60,40)*{\bullet}="5";
(70,50)*{\bullet}="6";
(80,20)*{\bullet}="7b";
(80,30)*{\bullet}="7a";
(80,40)*{\bullet}="7";
(80,50)*{\bullet}="8";
(80,60)*{\bullet}="9";
(100,30)*{\bullet}="13";
(100,20)*{\bullet}="13b";
(90,45)*{\bullet}="14";
(100,40)*{\bullet}="15";
(110,30)*{\bullet}="16";
"7b";"7a" **\crv{} ?>* \dir{-};
"7";"7a" **\crv{} ?>* \dir{-};
"3";"5" **\crv{} ?>* \dir{-};
"4";"5" **\crv{} ?>* \dir{-};
"5";"6" **\crv{} ?>* \dir{-};
"6";"9" **\crv{} ?>* \dir{-};
"7";"8" **\crv{} ?>* \dir{-};
"8";"9" **\crv{} ?>* \dir{-};
"13";"15" **\crv{} ?>* \dir{-};
"14";"15" **\crv{} ?>* \dir{-};
"15";"16" **\crv{} ?>* \dir{-};
"13b";"13" **\crv{} ?>* \dir{-};
"14";"8" **\crv{} ?>* \dir{-};
\endxy
\]
\caption{A connected partial transformation with a maximal left ray.}\label{f54}
\end{figure}

\begin{figure}[h]
\[
\xy
(60,30)*{\bullet}="3";
(50,30)*{\bullet}="4";
(60,40)*{\bullet}="5";
(70,50)*{\bullet}="6";
(80,40)*{\bullet}="7";
(80,30)*{\bullet}="7b";
(80,20)*{\bullet}="7c";
(70,20)*{\bullet}="7cc";
(80,10)*{\bullet}="7d";
(80,50)*{\bullet}="8";
(80,60)*{\bullet}="9";
(100,30)*{\bullet}="13";
(100,20)*{\bullet}="13b";
(90,45)*{\bullet}="14";
(100,40)*{\bullet}="15";
(110,30)*{\bullet}="16";
"3";"5" **\crv{} ?>* \dir{-};
"4";"5" **\crv{} ?>* \dir{-};
"5";"6" **\crv{} ?>* \dir{-};
"7d";"7c" **\crv{} ?>* \dir{-};
"7c";"7b" **\crv{} ?>* \dir{-};
"7cc";"7b" **\crv{} ?>* \dir{-};
"7b";"7" **\crv{} ?>* \dir{-};
"6";"9" **\crv{} ?>* \dir{-};
"7";"8" **\crv{} ?>* \dir{-};
"8";"9" **\crv{} ?>* \dir{-};
"13";"15" **\crv{} ?>* \dir{-};
"14";"15" **\crv{} ?>* \dir{-};
"15";"16" **\crv{} ?>* \dir{-};
"13b";"13" **\crv{} ?>* \dir{-};
"14";"8" **\crv{} ?>* \dir{-};
\endxy
\]
\caption{A connected partial transformation of type $\cho$.}\label{f55}
\end{figure}

For our purposes, it will not be necessary to distinguish connected partial transformations that have double rays only
or left rays only.
(In other words, if a connected $\gam\in P(X)$ has a double ray, then it will not matter whether it has a maximal right ray as well;
similarly, if it has a maximal left ray, then it will not matter whether it has a maximal chain as well.)
However, we will need to distinguish connected transformations that have right rays only,
and connected transformations that have chains only.

\begin{defi}\label{drro}
{\rm
Let $\gam\in P(X)$ be connected. If $\gam$ satisfies (3) of Remark~\ref{rtyp}, we
will say that $\gam$ is of (or has) \emph{type} $\rro$ (``right rays only'').
If $\gam$ satisfies (5) of Remark~\ref{rtyp}, we will say that $\gam$ is of type $\cho$ (``chains only'').
}
\end{defi}

\begin{lemma}\label{lcho}
Let $\gam\in P(X)$ be connected such that $\gam$ contains a maximal left ray or it is of type $\cho$.
Then $\gam$ contains a unique terminal vertex.
\end{lemma}
\begin{proof}
Since $\gam$ contains a maximal left ray or a maximal chain, it contains a terminal vertex. Suppose $x$ and $y$
are terminal vertices in $\gam$. Since $\gam$ is connected, $x\gam^k=y\gam^m\ne\mz$ for some $k,m\geq0$.
But since $x$ and $y$ are terminal, this is only possible when $k=m=0$. Thus $x=y$.
\end{proof}

\begin{defi}\label{droo}
{\rm
Let $\gam\in P(X)$ be connected such that $\gam$ has a maximal left ray or is of type $\cho$.
The unique terminal vertex of $\gam$ established
by Lemma~\ref{lcho} will be called the \emph{root} of~$\gam$.
}
\end{defi}

For integers $a$ and $b$, we write $a\,|\,b$ if $a$ divides $b$, that is, if $b=ak$ for some integer $k$.
For integers $a$ and $n$ with $n\geq1$, we denote by $\mo(a,n)$ the unique integer $r$ in $\{0,1,\ldots,n-1\}$ such that
$a\equiv r\pmod n$. We note that
\begin{equation}\label{emod}
\mo(a+1,n)=\left\{\begin{array}{ll}
\mo(a,n)+1&\mbox{if $\mo(a,n)\ne n-1$},\\
0&\mbox{if $\mo(a,n)=n-1$}.
\end{array}\right.
\end{equation}

\begin{prop}\label{pcyc}
Let $\gam,\del\in P(X)$ be connected such that $\gam$ has a cycle $(x_0\,x_1\ldots\,x_{k-1})$. Then
$\Gamma(\gam)$ is rp-homomorphic to $\Gamma(\del)$ if and only if $\del$
has a cycle $(y_0\,y_1\ldots\,y_{m-1})$ such that $m\,|\, k$.
\end{prop}
\begin{proof}
Suppose there is an rp-homomorphism $\phi$ from $\Gamma(\gam)$ to $\Gamma(\del)$.
Let $y_i=x_i\phi$ for $i=0,1,\ldots,k-1$. Then $y_0\ard y_1\ard\cdots\ard y_{k-1}\ard y_0$,
and so $y_0\del^k=y_0$. Let $m$ be the smallest integer in $\{1,2,\ldots,k\}$ such that $y_0\del^m=y_0$.
Then $(y_0\,y_1\ldots\,y_{m-1})$ is a cycle in $\del$. By the Division Algorithm, $k=mq+r$ for some
$q,r\in\mathbb N$ with $0\leq r<m$. Since $y_0\del^m=y_0$, we have $y_0\del^{mq}=y_0$, and so
$y_0=y_0\del^k=(y_0\del^{mq})\del^r=y_0\del^r$. Thus $r=0$ by the definition of $m$, and so $k=mq$,
that is, $m\,|\, k$.

Conversely suppose that $\del$ has a desired cycle.
We will define an rp-homomorphism $\phi$ from $\Gamma(\gam)$ to $\Gamma(\del)$
such that $\dom(\phi)=\dom(\gam)$ and $\ima(\phi)=\{y_0,y_1,\ldots,y_{m-1}\}$.
(Note that $\dom(\gam)=\spa(\gam)$ since $\gam$ has a cycle.)
For $x\in\dom(\gam)$, let
$p_x$ be the smallest nonnegative integer such that $x\gam^{p_x}=x_0$ (such $p_x$ exists by Lemma~\ref{lcom}),
and let $q_x=\mo(-p_x,m)$. Define $\phi$ on $\dom(\gam)$ by $x\phi=y_{q_x}$. Suppose $x\arga z$. We consider
two possible cases.
\vskip 1mm
\noindent\textbf{Case 1.} $x=x_0$.
\vskip 1mm
Then $p_x=0$, $z=x\gam=x_0\gam=x_1$, and $p_z=k-1$. Thus $q_x=\mo(0,m)=0$ and $q_z=\mo(-k+1,m)=1$
(since $m\,|\, k$, and so $-k\equiv0\pmod m$). Hence $x\phi=y_0\ard y_1=z\phi$.
\vskip 1mm
\noindent\textbf{Case 2.} $x\ne x_0$.
\vskip 1mm
Then, since $x\arga z$, we have $p_z=p_x-1$, and so
\begin{equation}\label{epcyc1}
q_z=\mo(-p_z,m)=\mo(-p_x+1,m).
\end{equation}

Suppose $q_x=\mo(-p_x,m)\ne m-1$. Then, by (\ref{emod}) and (\ref{epcyc1}), $q_z=\mo(-p_x+1,m)=\mo(-p_x,m)+1=q_x+1$, and so
$x\phi=y_{q_x}\ard y_{q_x+1}=y_{q_z}=z\phi$.

Suppose $q_x=\mo(-p_x,m)=m-1$. Then $-p_x\equiv-1\pmod m$, and so $p_x\equiv1\pmod m$. Thus $p_x=tm+1$
for some integer $t$, and so $p_z=p_x-1=tm$. Hence $q_z=\mo(-p_z,m)=\mo(-tm,m)=0$, and so
$x\phi=y_{q_x}=y_{m-1}\ard y_0=y_{q_z}=z\phi$.

Thus, in both cases, $x\phi\ard x\phi$, and so $\phi$ is an rp-homomorphism. (Condition (b) of Definition~\ref{drph}
is satisfied since $\Gamma(\gam)$ does not have any terminal vertices.)
\end{proof}

\begin{lemma}\label{lcyc}
Let $\gam,\del\in P(X)$ be connected such that $\del$ has a cycle $(y_0\,y_1\ldots\,y_{m-1})$.
Suppose $\gam$ has a double ray or $\gam$ is of type $\rro$. Then
$\Gamma(\gam)$ is rp-homomorphic to $\Gamma(\del)$.
\end{lemma}
\begin{proof}
Suppose $\gam$ has a double ray $\mu=\lan\ldots\,x_{-1}\,x_0\,x_1\ldots\ran$.
We will define an rp-homomorphism $\phi$ from $\Gamma(\gam)$ to $\Gamma(\del)$
such that $\dom(\phi)=\dom(\gam)$ and $\ima(\phi)=\{y_0,y_1,\ldots,y_{m-1}\}$. For $x\in\dom(\gam)$, let
$p_x$ be the smallest nonnegative integer such that $x\gam^{p_x}=x_i$ for some $i$ (such $p_x$ exists by Lemma~\ref{lcom}),
and let $q_x=\mo(i-p_x,m)$. Define $\phi$ on $\dom(\gam)$ by $x\phi=y_{q_x}$. Suppose $x\arga z$. We consider
two possible cases.
\vskip 1mm
\noindent\textbf{Case 1.} $x=x_i$ for some $i\in\mathbb Z$.
\vskip 1mm
Then $p_x=0$, $z=x\gam=x_i\gam=x_{i+1}$, and $p_z=0$. Thus $q_x=\mo(i,m)$ and $q_z=\mo(i+1,m)$.
If $q_x\ne m-1$, then $q_z=q_x+1$, and so $x\phi=y_{q_x}\ard y_{q_x+1}=y_{q_z}=z\phi$.
if $q_x=m-1$, then $q_z=0$, and so $x\phi=y_{q_x}=y_{m-1}\ard y_0=y_{q_z}=z\phi$.
\vskip 1mm
\noindent\textbf{Case 2.} $x\ne x_i$ for every $i\in\mathbb Z$.
\vskip 1mm
Then, since $x\arga z$, we have $p_z=p_x-1$ with $x\gam^{p_x}=z\gam^{p_z}=i$, and so
\begin{equation}\label{elcyc1}
q_z=\mo(i-p_z,m)=\mo(i-p_x+1,m).
\end{equation}
If $q_x\ne m-1$, then, by (\ref{emod}) and (\ref{elcyc1}), $q_z=\mo(i-p_x+1,m)=\mo(i-p_x,m)+1=q_x+1$, and so
$x\phi=y_{q_x}\ard y_{q_x+1}=y_{q_z}=z\phi$.
If $q_x=m-1$, then $q_z=0$, and so again $x\phi\ard z\phi$.

Hence, since $\Gamma(\gam)$ has no terminal vertices, $\phi$ is an rp-homomorphism.
The proof in the case when $\gam$ has type $\rro$ is similar.
\end{proof}

\begin{lemma}\label{ldch}
Let $\gam,\del\in P(X)$ be connected. Suppose that $\del$ has a double ray
and $\gam$ either has a double ray or has type $\rro$.
Then $\Gamma(\gam)$ is rp-homomorphic to $\Gamma(\del)$.
\end{lemma}
\begin{proof}
Suppose $\gam$ has a double ray, say
$\lan\ldots\,x_{-1}\,x_0\,x_1\ldots\ran$, and let
$\lan\ldots\,y_{-1}\,y_0\,y_1\ldots\ran$ be a double ray in $\del$.
We will define an rp-homomorphism $\phi$ from $\Gamma(\gam)$ to $\Gamma(\del)$
such that $\dom(\phi)=\dom(\gam)$ and $\ima(\phi)=\{\ldots,y_{-1},y_0,y_1,\ldots\}$.
(Note that $\dom(\gam)=\spa(\gam)$ since $\gam$ has a double ray or it is of type $\rro$.)
For $x\in\dom(\gam)$, let
$p_x$ be the smallest nonnegative integer such that $x\gam^{p_x}=x_i$ for some integer $i$.
Define $\phi$ on $\dom(\gam)$ by $x\phi=y_{i-p_x}$ where $x\gam^{p_x}=x_i$. Suppose $x\arga z$. We consider
two possible cases.
\vskip 1mm
\noindent\textbf{Case 1.} $x=x_i$ for some integer $i$.
\vskip 1mm
Then $p_x=0$, $z=x\gam=x_i\gam=x_{i+1}$, and $p_z=0$. Thus
\[
x\phi=y_{i-p_x}=y_i\ard y_{i+1}=y_{i+1-p_z}=z\phi.
\]
\noindent\textbf{Case 2.} $x\ne x_i$ for every integer $i$.
\vskip 1mm
Then, since $x\arga z$, we have $p_z=p_x-1$ and $x\gam^{p_x}=z\gam^{p_z}=x_i$ for some $i$. Thus
\[
x\phi=y_{i-p_x}\ard y_{i-p_x+1}=y_{i-p_z}=z\phi.
\]

Thus, in both cases, $x\phi\ard z\phi$, and so $\phi$ is an rp-homomorphism since $\Gamma(\gam)$ does not have any terminal vertices.
The proof in the case when $\gam$ has type $\rro$ is similar.
\end{proof}

\begin{lemma}\label{llra}
Let $\gam,\del\in P(X)$ be connected. Suppose that $\del$ has a maximal left ray
and $\gam$ either has a maximal left ray or is of type $\cho$.
Then $\Gamma(\gam)$ is rp-homomorphic to $\Gamma(\del)$.
\end{lemma}
\begin{proof}
Let $\lan\ldots\,y_2\,y_1\,y_0]$ be a maximal left ray in $\del$. Note that $y_0$ is the root of $\del$.
Let $x_0$ be the root of $\gam$.
We will define an rp-homomorphism $\phi$ from $\Gamma(\gam)$ to $\Gamma(\del)$
such that $\dom(\phi)=\spa(\gam)$ and $\ima(\phi)\subseteq\{\ldots,y_2,y_1,y_0\}$. For $x\in\spa(\gam)$, let
$p_x$ be the smallest nonnegative integer such that $x\gam^{p_x}=x_0$ (such $p_x$ exists by Lemma~\ref{lcom}).
Define $\phi$ on $\spa(\gam)$ by $x\phi=y_{p_x}$. If $x\arga z$, then $p_z=p_x-1$, and so
$x\phi=y_{p_x}\ard y_{p_x-1}=y_{p_z}=z\phi$.
Further, the only terminal vertex in $\Gamma(\gam)$ is $x_0$ and $x_0\phi=y_0$ (since $p_{x_0}=0$),
which is a terminal vertex in $\Gamma(\del)$. Hence $\phi$ is an rp-homomorphism.
\end{proof}

\begin{lemma}\label{lrronotlc}
Let $\gam,\del\in P(X)$ be connected such that $\gam$ is of type $\rro$. Suppose $\Gamma(\gam)$
is rp-homomorphic to $\Gamma(\del)$. Then $\del$ cannot have a maximal left ray or be of type $\cho$.
\end{lemma}
\begin{proof}
Let $\phi$ be an rp-homomorphism from $\Gamma(\gam)$ to $\Gamma(\del)$. Select a right ray
$[x_0\,x_1\,x_2\ldots\ran$ in $\gam$. Suppose to the contrary that $\del$ has a maximal left ray
or is of type $\cho$. Let $y_0$ be the root of $\del$. By Lemma~\ref{lcom},
$(x_0\phi)\del^k=y_0$ for some integer $k\geq0$. By Lemma~\ref{laggb}, $\gam\phi=\phi\del$, and so
$(x_0\phi)\del^{k+1}=(x_0\gam^{k+1})\phi=x_{k+1}\phi$.
But $(x_0\phi)\del^{k+1}=(x_0\phi)\del^k\del=y_0\del=\mz$, and so $x_{k+1}\phi=\mz$,
which is a contradiction. The result follows.
\end{proof}

\begin{prop}\label{pspl}
Let $S\leq P(X)$ such that $S$ is constant rich, and let $\al,\bt\in S$ with $\al\ne0$. Then there is
an rp-homomorphism $\phi\in S^1$ from $\Gamma(\al)$ to $\Gamma(\bt)$ with $\dom(\phi)=\spa(\al)$ if and only if
\begin{itemize}
  \item [\rm(a)] for every connected component $\gam$ of $\al$, there exist a connected component $\del$ of $\bt$
and an rp-homomorphism $\phi_\gam\in P(X)$ from $\Gamma(\gam)$ to $\Gamma(\del)$ with $\dom(\phi_\gam)=\spa(\gam)$; and
  \item [\rm(b)] $\join_{\gam\in C}\phi_\gam\in S^1$, where $C$ is the collection of connected components
of $\al$.
\end{itemize}
\end{prop}
\begin{proof}
Suppose there is an rp-homomorphism $\phi\in S^1$ from $\Gamma(\al)$ to $\Gamma(\bt)$
such that $\dom(\phi)=\spa(\al)$.
Let $\gam$ be a connected
component of $\al$ and let $x\in\spa(\gam)$. Then, by Proposition~\ref{pdec}, $x\phi\in\del$ for some
connected component $\del$ of $\bt$.
We claim that $(\spa(\gam))\phi\subseteq\spa(\del)$. Let $z\in\spa(\gam)$.
Since $\gam$ is connected, $x\al^k=x\gam^k=z\gam^m=z\al^m\ne\mz$ for some integers $k,m\geq0$.
By Lemma~\ref{laggb}, we have $\al\phi=\phi\bt$, and so $(z\phi)\bt^m=(z\al^m)\phi=(x\al^k)\phi=(x\phi)\bt^k\ne\mz$,
which implies that $z\phi$ and $x\phi$ are in the span of the same connected component of $\bt$,
that is, $z\phi\in\spa(\del)$. The claim has been proved.
Let $\phi_\gam=\phi|_{\spa(\gam)}$. Then $\phi_\gam$ is an rp-homomorphism from $\Gamma(\gam)$ to $\Gamma(\del)$
(by the claim and the fact that $\phi$ is an rp-homomorphism from $\Gamma(\al)$ to $\Gamma(\bt)$),
$\dom(\phi_\gam)=\spa(\gam)$ (by the definition of $\phi_\gam$), and $\join_{\gam\in C}\phi_\gam=\phi\in S^1$
(by the definition of $\phi_\gam$ and the fact that $\dom(\phi)=\spa(\al)$).

Conversely, suppose that (a) and (b) are satisfied. Let $\phi=\join_{\gam\in C}\phi_\gam$.
Note that $\phi$ is well defined since $\phi_\gam$ and $\phi_{\gam'}$ are disjoint
if $\gam\ne\gam'$.
Suppose $y\ara z$. Then $y,z\in\spa(\gam)$ for some connected component $\gam$ of $\al$. Thus $y,z\in\dom(\phi_\gam)$
and $y\phi=y\phi_\gam\ard z\phi_\gam=z\phi$, implying $y\phi\arb z\phi$. Suppose $y$ is a terminal vertex in $\Gamma(\al)$
and $y\in\dom(\phi)$. Then, there is a unique connected component $\gam$ of $\al$ such that $y$ is a terminal
vertex in $\Gamma(\gam)$. Then $y\phi=y\phi_\gam$ is a terminal vertex in $\Gamma(\del)$,
and so a terminal vertex in $\Gamma(\bt)$. Hence $\phi$ is an rp-homomorphism from $\Gamma(\al)$ to $\Gamma(\bt)$.
Moreover, $\dom(\phi)=\spa(\al)$ (by the definition of $\phi$) and $\phi\in\ S^1$ (by (b)).
\end{proof}

\begin{lemma}\label{lcom1}
Let $\al,\bt\in P(X)$ be such that $\Gamma(\al)$ is rp-homomorphic to $\Gamma(\bt)$. If $\al$ has a cycle
of length $k$, then $\bt$ has a cycle of length $m$ such that $m\,|\, k$.
\end{lemma}
\begin{proof}
It follows immediately from Propositions~\ref{pcyc} and~\ref{pspl}.
\end{proof}

A binary relation $R$ on a set $A$ is called \emph{well founded} if every nonempty
subset $B\subseteq A$ contains an $R$-minimal element; that is, $a\in B$ exists such that there is no $y\in B$
with $(y,a)\in R$ \cite[page~25]{Je06}. Let $R$ be a well-founded relation on $A$. Then there is a unique
function $\rho$ defined on $A$ with ordinals as values such that for every $x\in A$,
\begin{equation}\label{ewel}
\rho(x)=\sup\{\rho(y)+1:(y,x)\in R\}.
\end{equation}
The ordinal $\rho(x)$ is called the \emph{rank} of $x$ in $\lan A,R\ran$ \cite[Theorem~2.27]{Je06}.

For a mapping $f:A\to B$ and $b\in B$, we denote by $bf^{-1}$ the preimage of $b$ under $f$.

\begin{defi}\label{dord}
{\rm
Let $\gam\in P(X)$ be connected of type $\rro$ or $\cho$. Recall that $R_\gam$ is a binary relation on $\spa(\gam)$
defined by $(y,x)\in R_\gam$ if $y\gam=x$. (Note that $(y,x)\in R_\gam\miff y\arga x\miff y\in x\gam^{-1}$).
The relation  $R_\gam$ is well founded since there is no sequence
$\lan x_0,x_1,x_2,\ldots\ran$ such that $\cdots\arga x_2\arga x_1\arga x_0$.
For $x\in\spa(\gam)$,we will denote the rank of $x$ in $\lan\spa(\gam),R_\gam\ran$ by $\rho_\gam(x)$
(or $\rho(x)$ if $\gam$ is clear from the context).

It follows from (\ref{ewel}) that for every $x\in\spa(\gam)$ with $\rho(x)>0$, we have
$\rho(y)<\rho(x)$ for every $y\in x\gam^{-1}$, and if $\nu=\sup\{\rho(y):y\in x\gam^{-1}\}$ then
\begin{equation}\label{edord1}
\rho(x)=\left\{\begin{array}{ll}
\nu+1&\mbox{if $\rho(y)=\nu$ for some $y\in x\gam^{-1}$},\\
\nu&\mbox{if $\rho(y)<\nu$ for every $y\in x\gam^{-1}$}.
\end{array}\right.
\end{equation}
(Indeed, suppose that $\rho(y_0)=\nu$ for some $y_0\in x\gam^{-1}$. Then
$\rho(y)\leq\rho(y_0)$ for all $y\in x\gam^{-1}$. Thus $\rho(y)+1\leq\rho(y_0)+1$ for all $y\in x\gam^{-1}$, and so
$\rho(x)=\sup\{\rho(y)+1:y\in x\gam^{-1}\}=\rho(y_0)+1=\nu+1$.
Now suppose that $\rho(y)<\nu$ for every $y\in x\gam^{-1}$. Then $\rho(y)+1\leq\nu$ for every $y\in x\gam^{-1}$,
and so $\nu=\sup\{\rho(y):y\in x\gam^{-1}\}\leq\sup\{\rho(y)+1:y\in x\gam^{-1}\}\leq\nu$. Thus $\rho(x)=\sup\{\rho(y)+1:y\in x\gam^{-1}\}=\nu$.)
}
\end{defi}

\begin{example}\label{eord}
{\rm
Let $X=\{x_0,x_1,x_2,\ldots,y_0,y_1,y_2,\ldots\}$ and let
\[
\gam=[x_0\,x_1\,x_2\,x_3\ldots\ran\jo[y_0\,x_2]\jo[y_1\,y_2\,x_2]\jo[y_3\,y_4\,y_5\,x_2]\jo[y_6\,y_7\,y_8\,y_9\,x_2]\jo\cdots\in P(X).
\]
Then $\gam$ is connected of type $\rro$ and we have: $\rho(x_0)=0$, $\rho(x_1)=1$, and $\rho(x_{2+i})=\ome+i$ for every $i\geq0$,
where $\omega$ is the smallest infinite ordinal.
We also have: $\rho(y_6)=0$, $\rho(y_7)=1$, $\rho(y_8)=2$, and $\rho(y_9)=3$.
}
\end{example}

\begin{example}\label{eord2}
{\rm
Let $X=\{y_0,y_1,y_2,\ldots\}\cup\bigcup_{i=0}^{\infty}\{z_0^i,z_1^i,z_2^i,\ldots\}$.
For every integer $i\geq0$, let
\[
\del_i=[z_1^i\,z_0^i]\jo[z_2^i\,z_3^i\,z_0^i]\jo[z_4^i\,z_5^i\,z_6^i\,z_0^i]\jo[z_7^i\,z_8^i\,z_9^i\,z_{10}^i\,z_0^i]\jo\cdots\in P(X).
\]
Then each $\del_i$ is connected of type $\cho$ and $\rho_{\del_i}(z_0^i)=\ome$. Further, let
\[
\gam=(\del_0\jo[z_0^0\,y_0])\jo(\del_1\jo[z_0^1\,y_1\,y_0])\jo(\del_2\jo[z_0^2\,y_2\,y_3\,y_0])\jo(\del_3\jo[z_0^3\,y_4\,y_5\,y_6\,y_0])\jo\cdots\in P(X).
\]
Then $\gam$ is connected of type $\cho$ and $\rho_\gam(y_0)=\ome+\ome=2\ome$.
}
\end{example}

We will need the following lemma from the theory of well-founded relations \cite[Appendix~B]{Ke95}.

\begin{lemma}\label{lapb}
Let $R_1$ and $R_2$ be well-founded relations on $A_1$ and $A_2$, respectively. Suppose a function
$f:A_1\to A_2$ is such that for all $x,y\in A_1$, if $(x,y)\in R_1$ then $(xf,yf)\in R_2$.
The for all $x\in A_1$, $\rho(x)\leq\rho(xf)$.
\end{lemma}

\begin{nota}\label{ndax}
{\rm
Let $\gam\in P(X)$ be connected and let $x\in\spa(\gam)$. We denote by $\dax$ the set of all $y\in\spa(\gam)$
such that $x=y\gam^m$ for some $m\geq0$. If $x\in\ima(\gam)$, we denote by $\gam_x$ the restriction
of $\gam$ to $\dax\sm\{x\}$. Note that $\gam_x$ is connected and it either contains a maximal left ray or is of type $\cho$,
and that, in either case, $x$ is the root of $\gam_x$.
}
\end{nota}

\begin{lemma}\label{linc}
Let $\gam,\del\in P(X)$ be connected such that $\gam$ is of type $\rro$ or $\cho$ and $\del$ is contained in $\gam$. Then
for every $x\in\spa(\del)$:
\begin{itemize}
  \item [\rm(1)] $\rho_\del(x)\leq \rho_\gam(x)$;
  \item [\rm(2)] if $\del=\gam_z$ for some $z\in\ima(\gam)$, then $\rho_\del(x)=\rho_\gam(x)$.
\end{itemize}
\end{lemma}
\begin{proof}
First note that $\del$ must be of type $\rro$ or $\cho$. Statement~(1) follows from Lemma~\ref{lapb}
with $f:\spa(\del)\to\spa(\gam)$ defined by $xf=x$. To prove (2), we suppose $\del=\gam_z$ and proceed by
well-founded induction \cite[Theorem~2.6]{Je06}. Let $x\in\spa(\del)$.
The result is true if $x$ is $R_\del$-minimal since then $x$ is also $R_\gam$-minimal.
Suppose $\rho_\del(y)=\rho_\gam(y)$ for all $y\in\spa(\del)$ such that $(y,x)\in R_\del$. Then
\[
\rho_\del(x)=\sup\{\rho_\del(y)+1:(y,x)\in R_\del\}=\sup\{\rho_\gam(y)+1:(y,x)\in R_\gam\}=\rho_\gam(x),
\]
where the last but one equality follows from the inductive hypothesis and the fact that for $\del=\gam_z$,
$x\del^{-1}=x\gam^{-1}$ for all $x\in\spa(\del)$.
\end{proof}

\begin{prop}\label{pcho}
Let $\gam,\del\in P(X)$ be connected of type $\cho$ with roots $x_0$ and $y_0$, respectively. Then
$\Gamma(\gam)$ is rp-homomorphic to $\Gamma(\del)$ if and only if $\rho(x_0)\leq \rho(y_0)$.
\end{prop}
\begin{proof}
Suppose there is an rp-homomorphism $\phi$ from $\Gamma(\gam)$ to $\Gamma(\del)$. Then $\phi:\spa(\gam)\to\spa(\del)$,
$x_0\phi=y_0$, and
for all $x,y\in\spa(\gam)$, if $x\arga y$ then $x\phi\ard y\phi$. Thus $\rho(x_0)\leq \rho(y_0)$ by Lemma~\ref{lapb}.

Conversely, suppose $\rho(x_0)\leq \rho(y_0)$. We will prove that $\Gamma(\gam)$ is rp-homomorphic to $\Gamma(\del)$
by transfinite induction on $\rho(x_0)$. Let $\rho(x_0)=1$. Then for every $z\in\dom(\gam)$, we have $z\arga x_0$.
Since $\rho(y_0)\geq \rho(x_0)=1$, there is some $w\in\dom(\del)$ such that $w\ard y_0$. Define $\phi$ on $\spa(\gam)$ by:
$x_0\phi=y_0$ and $z\phi=w$ for every $z\in\dom(\gam)$. Then clearly $\phi$ is an rp-homomorphism
from $\Gamma(\gam)$ to $\Gamma(\del)$.

Let $\rho(x_0)=\mu>1$ and suppose that for all connected $\gam_1,\del_1\in P(X)$ of type $\cho$ with roots $z$ and $w$,
respectively, if $\rho(z)<\mu$ and $\rho(z)\leq \rho(w)$, then $\Gamma(\gam_1)$ is rp-homomorphic to $\Gamma(\del_1)$.

Let $z\in x_0\gam^{-1}$ and note that $\rho(z)<\mu$. Since $\rho(y_0)\geq\mu$, there is
$w_z\in y_0\del^{-1}$ such that $\rho(z)\leq \rho(w_z)$. If $z\in\ima(\gam)$, then $\gam_z$ and $\del_{w_z}$
are connected with $\rho_{\gam_z}(z)=\rho(z)\leq \rho(w_z)=\rho_{\del_{w_z}}(w_z)$, and so, by the inductive hypothesis, there is an rp-homomorphism
$\phi_z$ from $\Gamma(\gam_z)$ to $\Gamma(\del_{w_z})$.
If $z\notin\ima(\gam)$ (that is, if $\daz=\{z\}$),
we define $\phi_z$ on $\daz=\{z\}$ by $z\phi_z=w_z$.

Define $\phi$ on $\spa(\gam)$ by: $x_0\phi=y_0$ and $u\phi=u\phi_z$ if $u\in\,\daz$ for some $z\in x_0\gam^{-1}$.
Then $\phi$ is well-defined since the collection $\{\daz\}_{z\in x_0\gam^{-1}}$ is a partition of
$\dom(\gam)$ ($=\spa(\gam)\sm\{x_0\}$). Suppose $u\arga v$. If $v\in\,\daz$ for some $z\in x_0\gam^{-1}$, then $u\in\,\daz$ as well, and so
$u\phi=u\phi_z\ard v\phi_z=v\phi$. If $v=x_0$, then $u=z\in x_0\gam^{-1}$, and so $u\phi=z\phi=z\phi_z=w_z\ard y_0=x_0\phi=v\phi$.
Hence, since $x_0\phi=y_0$ and $x_0$ is the unique terminal vertex of $\Gamma(\gam)$,
$\phi$ is an rp-homomorphism from $\Gamma(\gam)$ to $\Gamma(\del)$.
\end{proof}

\begin{defi}\label{ddom}
{\rm
Let $\lan a_n\ran_{n\geq0}$ and $\lan b_n\ran_{n\geq0}$ be sequences of ordinals (indexed by nonnegative integers $n$).
We say that $\lan b_n\ran$ \emph{dominates} $\lan a_n\ran$ if there is $k\geq0$ such that
\[
b_{k+n}\geq a_n\mbox{ for every $n\geq0$}.
\]
}
\end{defi}

\begin{nota}\label{nseq}
{\rm
Let $\gam\in P(X)$ be connected of type $\rro$ and let $\eta=[x_0\,x_1\,x_2\ldots\ran$ be a maximal right ray in $\gam$.
We denote by $\lan \eta^\gam_n\ran_{n\geq0}$ the sequence of ordinals such that
\[
\eta^\gam_n=\rho_\gam(x_n)\mbox{ for every $n\geq0$}.
\]
}
\end{nota}

For example, for $\gam$ from Example~\ref{eord} and the right ray $\eta=[x_0\,x_1\,x_2\ldots\ran$ in $\gam$,
the sequence $\lan \eta^\gam_n\ran$ is $\lan0,\,1,\,\omega,\,\omega+1,\,\omega+2,\,\omega+3,\ldots\ran$.

\begin{prop}\label{prro}
Let $\gam,\del\in P(X)$ be connected of type $\rro$. Then
$\Gamma(\gam)$ is rp-homomorphic to $\Gamma(\del)$ if and only if there are maximal right rays $\eta$ in $\gam$
and $\xi$ in $\del$ such that $\lan \xi^\del_n\ran$ dominates~$\lan \eta^\gam_n\ran$.
\end{prop}
\begin{proof}
Suppose there is an rp-homomorphism $\phi$ from $\Gamma(\gam)$ to $\Gamma(\del)$.
Select a maximal right ray $\eta=[x_0\,x_1\,x_2\ldots\ran$ in $\gam$ (possible by Proposition~\ref{pcom}.)
Then $x_0\phi\ard x_1\phi\ard x_2\phi\ard\cdots$, and so, since $\del$ does not have any double rays,
there is $w\in\dom(\del)-\ima(\del)$ such that $w\del^k=x_0\phi$ for some $k\geq0$. Thus
\[
\xi=[y_0=w\,\,\,y_1=w\del\,\ldots\,y_{k-1}=w\del^{k-1}\,\,\,y_k=w\del^k=x_0\phi\,\,\,y_{k+1}=x_1\phi\,\,\,y_{k+2}=x_2\phi\,\ldots\ran
\]
is a maximal right ray in $\del$. For every $n\geq0$, the mapping $\phi|_{\downarrow x_n}$ is an rp-homomorphism
from $\Gamma(\gam_{x_n})$ to $\Gamma(\del_{y_{k+n}})$ (see Notation~\ref{ndax}).
Thus for every $n\geq0$, we have $\rho_{\gam_{x_n}}(x_n)\leq \rho_{\del_{y_{k+n}}}(y_{k+n})$
by Proposition~\ref{pcho}, and so $\rho(x_n)\leq \rho(y_{k+n})$ by Lemma~\ref{linc}. Hence $\lan \xi^\del_n\ran$ dominates $\lan \eta^\gam_n\ran$.

Conversely, suppose there are maximal right rays $\eta=[x_0\,x_1\,x_2\ldots\ran$ in $\gam$ and
$\xi=[y_0\,y_1\,y_2\ldots\ran$ in $\del$
such that $\lan \xi^\del_n\ran$ dominates $\lan \eta^\gam_n\ran$,
that is, there is $k\geq0$ such that $\xi^\del_{k+n}\geq\eta^\gam_n$ for every $n\geq0$. We define a collection $\{B_n\}_{n\geq0}$
of subsets of $\spa(\gam)$ by
\[
B_0=\{x_0\},\,\,\,B_n=\downarrow\!\!x_n-\downarrow\!\!x_{n-1}\mbox{ for $n\geq1$}.
\]
Since $\gam$ is connected, $\{B_n\}_{n\geq0}$ is a partition of $\spa(\gam)$.

We will now define an rp-homomorphism $\phi$ from $\Gamma(\gam)$ to $\Gamma(\del)$ by defining $\phi$ on $B_n$
for every $n\geq0$. First, we set $x_0\phi=y_k$. Let $n\geq1$. If $B_n=\{x_n\}$, we set $x_n\phi=y_{k+n}$.
Suppose $|B_n|\geq2$. Let $\gam_n=\gam|_{B_n\sm\{x_n\}}$ and $\del_n=\del_{y_{k+n}}$. Then $\gam_n$
and $\del_n$ are connected of type $\cho$ with roots $x_n$ and $y_{k+n}$, respectively.
By Lemma~\ref{linc},
\[
\rho_{\gam_n}(x_n)\leq \rho_\gam(x_n)=\eta^\gam_n\leq\xi^\del_{k+n}=\rho_\del(y_{k+n})=\rho_{\del_n}(y_{k+n}).
\]
Thus, by Proposition~\ref{pcho}, there is an rp-homomorphism $\phi_n$ from $\Gamma(\gam_n)$ to $\Gamma(\del_n)$.
Note that $x_n\phi_n=y_{k+n}$. We define $\phi$ on $B_n$ by $x\phi=x\phi_n$.

Suppose $x\arga z$. Then $z\in B_n$ for some $n\geq0$. If $x\in B_n$, then
$x\phi=x\phi_n\ard z\phi_n=z\phi$
since $\phi_n$ is an rp-homomorphism from $\Gamma(\gam_n)$ to $\Gamma(\del_n)$. If $x\notin B_n$, then
we must have $x=x_{n-1}$ and $z=x_n$, and so
$x\phi=x_{n-1}\phi=y_{k+n-1}\ard y_{k+n}=x_n\phi=z\phi$.

Hence, in all cases, if $x\arga z$ then $x\phi\ard z\phi$. Thus, since $\Gamma(\gam)$ does not have any terminal vertices,
$\phi$ is an rp-homomorphism from $\Gamma(\gam)$ to $\Gamma(\del)$.
\end{proof}

The following lemma will be needed in the next section.

\begin{lemma}\label{lct5}
Let $\gam,\del\in P(X)$ be of type $\rro$. Let $\eta$ be a maximal right ray in $\gam$ and
$\xi$ be a maximal right ray in $\del$ such that $\lan\xi^\del_n\ran$ dominates $\lan\eta^\gam_n\ran$.
Then for every maximal right ray $\eta_1$ in $\gam$ and every maximal right ray $\xi_1$ in $\del$
$\lan(\xi_1)^\del_n\ran$ dominates $\lan(\eta_1)^\gam_n\ran$.
\end{lemma}
\begin{proof}
Since $\lan\xi^\del_n\ran$ dominates $\lan\eta^\gam_n\ran$, there is an integer $k\geq0$ such that
\[
\xi^{\del}_{k+n}\geq\eta^\gam_n\mbox{ for every $n\geq0$.}
\]
Let $\eta=[x_0\,x_1\,x_2\ldots\ran$ and $\xi=[y_0\,y_1\,y_2\ldots\ran$.
Let $\eta_1=[w_0\,w_1\,w_2\ldots\ran$ and $\xi_1=[z_0\,z_1\,z_2\ldots\ran$ be arbitrary maximal
right rays in $\gam$ and $\del$, respectively.
Since $\gam$ and $\del$ are connected,
there are integers $l,q,m,p\geq0$ such that $x_l=x_0\gam^l=w_0\gam^q=w_q$ and
$y_m=y_0\del^m=z_0\del^p=z_p$. We may assume that
$m\geq k$.
Then for every $n\geq0$,
\begin{align}
(\xi_1)^\del_{(p+l)+n}&=\rho_\del(z_{p+(l+n)})=
\rho_\del(y_{m+(l+n)})\geq \rho_\del(y_{k+(l+n)})=\xi^\del_{k+(l+n)}\geq\eta^\gam_{l+n},\mbox{ and}\notag\\
\eta^\gam_{l+n}&=\rho_\gam(x_{l+n})=\rho_\gam(w_{q+n})\geq \rho_\gam(w_n)=(\eta_1)^\gam_n.\notag
\end{align}
Hence $\lan(\xi_1)^\del_n\ran$ dominates $\lan(\eta_1)^\gam_n\ran$.
\end{proof}

\section{Conjugacy in $P(X)$}\label{spx}
\setcounter{equation}{0}
\setcounter{figure}{0}

In this section we characterize the conjugacy $\con$ in the semigroup $P(X)$ of partial transformations on any
nonempty set $X$ (finite or infinite).

In $P(X)$ and, more generally, in any a constant rich subsemigroup $S$ of $P(X)$, 
the conjugacy relation $\con$ can be reformulated, as a consequence of Lemma~\ref{lgpa}(1), in the following way: 
given any $\al,\bt\in S$, we have
$\al\!\con\!\bt$ in $S$ if and only if there exist $\phi,\psi\in S^1$ such that $\al\phi=\phi\bt$ and $\bt\psi=\psi\al$, 
with $\dom(\al\phi)=\dom(\al)$ and $\dom(\bt\psi)=\dom(\bt)$. Notice that
the semigroup $P(X)$ can be regarded as a left restriction semigroup with respect to the set of partial 
identities $E=\{\id_Y:Y\subseteq X\}$ (see \cite{Hol09} for a survey). 
Hence $P(X)$ is equipped with a unary operation $^+$ assigning to any 
$\al\in P(X)$ the element $\al^+=\id_{\dom(\al)}$. Any subsemigroup $S$ of $P(X)$ closed under $^+$  is called a \emph{left restriction} semigroup. 
If $S$ is a left restriction semigroup that is also constant rich, then for all $\al,\bt\in S$,
\[\al\!\con\!\bt\ \Leftrightarrow\ \exists\phi,\psi\in S^1: \al\phi=\phi\bt\ {\rm and}\ \bt\psi=\psi\al, 
                                     {\rm with}\ (\al\phi)^+=\al^+\ {\rm and}\ (\bt\psi)^+=\bt^+. 
\]

We now proceed to characterize the conjugacy relation $\con$ in $P(X)$ in terms of the basic partial transformations.

\begin{defi}\label{dsac}
{\rm
Let $M$ be a nonempty subset of the set $\mathbb Z_+$ of positive integers. Then $M$ is
partially ordered by the relation $\mid$ (divides).
Order the elements of $M$ according to the usual ``less than'' relation:
$m_1<m_2<m_3<\ldots$. We define a subset $\sac(M)$ of $M$ as follows: for every integer $n$, $1\leq n<|M|+1$,
\[
m_n\in\sac(M)\miff (\forall_{i<n})\mbox{$m_n$ is not a multiple of $m_i$}.
\]
The set $\sac(M)$ is a maximal antichain of the poset $(M,\,\mid)$. We will call $\sac(M)$ the \emph{standard antichain of $M$}.
}
\end{defi}

For example, if $M=\{4,6,8,10,18\}$ then $\sac(M)=\{4,6,10\}$; if $M=\{1,2,4,8,16,32,\ldots\}$ then $\sac(M)=\{1\}$.

\begin{defi}\label{dcs}
{\rm
Let $\al\in P(X)$ such that $\al$ contains a cycle. Let
\[
M=\{n\in\mathbb Z_+:(\exists_{x\in\dom(\al)})\ x\al^n =x\mbox{ and $x\al^i\ne x$ for every $i, 1\leq i<n$}\}.
\]
Note that $M$ is the set of the lengths of cycles in $\al$.
The standard antichain of $(M,\,\mid)$ will be called the \emph{cycle set} of $\al$ and denoted by $\cs(\al)$.
We agree that $\cs(\al)=\emptyset$ if $\al$ has no cycles.
}
\end{defi}

\begin{theorem}\label{tpcha}
Let $\al,\bt\in P(X)$. Then $\al\!\con\!\bt$ in $P(X)$ if and only if $\al=\bt=0$
or $\al,\bt\ne0$ and the following conditions are satisfied:
\begin{itemize}
  \item [\rm(1)] $\cs(\al)=\cs(\bt)$;
  \item [\rm(2)] $\al$ has a double ray but not a cycle $\miff$ $\bt$ has a double ray but not a cycle;
  \item [\rm(3a)] if $\al$ has a connected component $\gam$ of type $\rro$, but no cycles or double rays, then
$\bt$ has a connected component $\del$ of type $\rro$, but no cycles or double rays, and
$\lan \xi^\del_n\ran$ dominates $\lan \eta^\gam_n\ran$
for some maximal right rays $\eta$ in $\gam$ and $\xi$ in $\del$;
  \item [\rm(3b)] if $\bt$ has a connected component $\del$ of type $\rro$, but no cycles or double rays, then
$\al$ has a connected component $\gam$ of type $\rro$, but no cycles or double rays, and
$\lan \eta^\gam_n\ran$ dominates $\lan \xi^\del_n\ran$
for some maximal right rays $\xi$ in $\del$ and $\eta$ in $\gam$;
  \item [\rm(4)] $\al$ has a maximal left ray $\miff$ $\bt$ has a maximal left ray;
  \item [\rm(5a)] if $\al$ has a connected component $\gam$ of type $\cho$ with root $x_0$, but no maximal left rays, then
$\bt$ has a connected component $\del$ of type $\cho$ with root $y_0$, but no maximal left rays, and  $\rho_\gam(x_0)\leq \rho_\del(y_0)$;
  \item [\rm(5b)] if $\bt$ has a connected component $\del$ of type $\cho$ with root $y_0$, but no maximal left rays, then
$\al$ has a connected component $\gam$ of type $\cho$ with root $x_0$, but no maximal left rays, and  $\rho_\del(y_0)\leq \rho_\gam(x_0)$.
\end{itemize}
\end{theorem}
\begin{proof}
Suppose $\al\!\con\!\bt$. Then, since $[0]_{\con}=\{0\}$ in every semigroup with $0$,
either $\al=\bt=0$ or $\al,\bt\ne0$. Suppose $\al,\bt\ne0$. Then,
by Theorem~\ref{tconrph}, there is an rp-homomorphism $\phi$ from $\Gamma(\al)$ to $\Gamma(\bt)$.
We may assume that $\dom(\phi)=\spa(\al)$ (see Remark~\ref{riso}).

Suppose $\al$ has a cycle.
Then, by Lemma~\ref{lcom1}, $\bt$ also has a cycle.
Let $n\in\cs(\al)$. Then $\al$ has a cycle of length $n$, and so $\bt$ has a cycle of length $m$ such that $m\,|\, n$.
By the definition of $\cs(\bt)$, there is $m_1\in\cs(\bt)$ such that $m_1\,|\, m$. Thus $\bt$ has a cycle of length $m_1$, and so
$\al$ has a cycle of length $n_1$ such that $n_1\,|\, m_1$, so $n_1\,|\, m_1\,|\, m\,|\, n$. Since $\cs(\al)$ is an antichain,
$n_1\,|\, n$ and $n\in\cs(\al)$ implies $n_1=n$. Thus $n=m_1$, and so $n\in\cs(\bt)$. We have proved that $\cs(\al)\subseteq\cs(\bt)$.
Similarly, $\cs(\bt)\subseteq\cs(\al)$, and so $\cs(\al)=\cs(\bt)$.
By symmetry, if $\bt$ has a cycle, then
$\al$ also has a cycle and $\cs(\bt)=\cs(\al)$. If neither $\al$ nor $\bt$ has a cycle, then $\cs(\al)=\cs(\bt)=\emptyset$.
We have proved (1).

Suppose $\al$ has a double ray, say $\lan\ldots\,x_{-1}\,x_0\,x_1\ldots\ran$, but no cycles. Then
$\bt$ does not have a cycle either by Lemma~\ref{lcom1}, and
$\ldots\arb x_{-1}\phi\arb x_0\phi\arb x_1\phi\arb\ldots$,
where $\phi$ is an rp-homomorphism from $\Gamma(\al)$ to $\Gamma(\bt)$.
The elements $\ldots,x_{-1}\phi,x_0\phi,x_1\phi,\ldots$ are pairwise disjoint (since otherwise
$\bt$ would have a cycle), and so
$\lan\ldots\,x_{-1}\phi\,\,x_0\phi\,\,x_1\phi\ldots\ran$ is a double ray in~$\bt$.
The converse is true by symmetry. This proves (2).

Suppose that $\al$ has a connected component $\gam$ of type $\rro$, but neither a cycle nor a double ray.
By Proposition~\ref{pspl},
there is a connected component $\del$ of $\bt$ such that $\Gamma(\gam)$ is rp-homomorphic to $\Gamma(\del)$. By (1) and (2),
$\del$ does not have a cycle or a double ray. By Lemma~\ref{lrronotlc}, $\del$
does not have a maximal left ray and it is not of type $\cho$. Hence $\del$ has type $\rro$.
By Proposition~\ref{prro}, there are maximal right rays $\eta$ in $\gam$ and $\xi$ in $\del$
such that $\lan \xi^\del_n\ran$ dominates $\lan \eta^\gam_n\ran$.
We have proved (3a). Condition (3b) holds by symmetry.

Suppose $\al$ has a maximal left ray, say $\lan\ldots\,x_2\,x_1\,x_0]$. Then
$\ldots\arb x_2\phi\arb x_1\phi\arb x_0\phi$
and $x_0\phi$ is a terminal vertex in $\Gamma(\bt)$, which implies that $\lan\ldots\,x_2\phi\,x_1\phi\,x_0\phi]$
is a maximal left ray in $\bt$. The converse is true by symmetry. This proves (4).

Suppose $\al$ has a connected component $\gam$ of type $\cho$ with root $x_0$, but not a maximal left ray.
By Proposition~\ref{pspl} and its proof,
there is a connected component $\del$ of $\bt$ such that $\phi_\gam=\phi|_{\spa(\gam)}$ is an
rp-homomorphism from $\Gamma(\gam)$ to $\Gamma(\del)$.
Since $x_0$ is a terminal vertex in $\gam$,
$y_0=x_0\phi_\gam$ is a terminal vertex in $\del$. Since $\bt$ has no maximal left ray (by (3)),
$\del$ is of type $\cho$ and $y_0$ is the root of $\del$.
By Proposition~\ref{pcho}, $\rho_\gam(x_0)\leq \rho_\del(y_0)$.
We have proved (5a). Condition (5b) holds by symmetry.

Conversely, if $\al=\bt=0$ then $\al\con\bt$. Suppose that $\al,\bt\ne0$ and that (1)--(5b) hold.
Let $\gam$ be a connected component of $\al$. We will
prove that $\Gamma(\gam)$ is rp-homomorphic to $\Gamma(\del)$ for some connected component $\del$ of $\bt$.

Suppose $\gam$ has a cycle of length $k$. Since, by (1), $\cs(\al)=\cs(\bt)$, $\bt$ has a cycle $\vth$ of length
$m$ such that $m\,|\, k$. Let $\del$ be the connected component of $\bt$ containing $\vth$. Then $\Gamma(\gam)$ is rp-homomorphic
to $\Gamma(\del)$ by Proposition~\ref{pcyc}.

Suppose $\gam$ has a double ray. If some connected component $\del$ of $\bt$ has a cycle, then
$\Gamma(\gam)$ is rp-homomorphic to $\Gamma(\del)$ by Lemma~\ref{lcyc}. Suppose $\bt$ does not have a cycle.
Then, by (1) and (2), both $\al$ and $\bt$ have a double ray but not a cycle. Let $\del$ be a connected component
of $\bt$ containing a double ray. Then $\Gamma(\gam)$ is rp-homomorphic to $\Gamma(\del)$ by Lemma~\ref{ldch}.

Suppose $\gam$ is of type $\rro$. If $\bt$ has some connected component $\del$ with a cycle or a double ray, then
$\Gamma(\gam)$ is rp-homomorphic to $\Gamma(\del)$ by Lemmas~\ref{lcyc} and~\ref{ldch}. Suppose $\bt$ does not have a cycle
or a double ray. Then, by (3a),
there is a connected component $\del$ in $\bt$ of type $\rro$ such that
$\lan \xi^\del_n\ran$ dominates $\lan \eta^\gam_n\ran$
for some maximal right rays $\eta$ in $\gam$ and $\xi$ in $\del$. Hence $\Gamma(\gam)$
is rp-homomorphic to $\Gamma(\del)$ by Proposition~\ref{prro}.

Suppose $\gam$ has a maximal left ray.
Then, by (4), some connected component $\del$
of $\bt$ has a maximal left ray. Then $\Gamma(\gam)$ is rp-homomorphic to $\Gamma(\del)$ by Lemma~\ref{llra}.

Suppose $\gam$ is of type $\cho$ with root $x_0$. If $\bt$ has some connected component $\del$ with a maximal left ray, then
$\Gamma(\gam)$ is rp-homomorphic to $\Gamma(\del)$ by Lemma~\ref{llra}. Suppose $\bt$ does not have
a maximal left ray. Then, by (4), $\al$ does not have a maximal left ray, and so, by (5a),
there is a connected component $\del$ in $\bt$ of type $\cho$ with root $y_0$ such that $\rho_\gam(x_0)\leq \rho_\del(y_0)$.
Hence $\Gamma(\gam)$
is rp-homomorphic to $\Gamma(\del)$ by Proposition~\ref{pcho}.

We have proved that for every connected component $\gam$ of $\al$, there exists a connected component $\del$ of $\bt$
and an rp-homomorphism $\phi_\gam\in P(X)$ from $\Gamma(\gam)$ to $\Gamma(\del)$. We may assume
that for every $\gam\in C(\al)$, $\dom(\phi_\gam)=\spa(\gam)$.
Hence $\Gamma(\al)$ is rp-homomorphic to $\Gamma(\bt)$ by Proposition~\ref{pspl}. By symmetry,
$\Gamma(\bt)$ is rp-homomorphic to $\Gamma(\al)$, and so $\al\con\bt$ by Theorem~\ref{tconrph}.
\end{proof}

\begin{example}\label{epcha}
{\rm
Let $X$ be an infinite set containing $x_0,y_1,y_2,y_3,\ldots$ and let $\al,\bt\in P(X)$ be the partial
transformations whose digraphs are presented in Figure~\ref{f71}. Then $\al$ is connected of type $\cho$ with root $x_0$,
and $\bt=\del_1\jo\del_2\jo\del_3\jo\del_4\jo\cdots$, where $\del_i$ is a chain with root $y_i$.
We have $\rho_\gam(x_0)=\ome$, where $\gam=\al$, and for every integer $i\geq1$, $\rho_{\del_i}(y_i)=i$. Hence
$\al$ and $\bt$ are not conjugate by (5a) of Theorem~\ref{tpcha}.
}
\end{example}

\begin{figure}[h]
\[
\xy
(10,42)*{x_{0}}="g";
(10,0)*{\Gamma(\alpha)}="g";
(32.5,20)*{\ldots}="h";
(0,30)*{\bullet}="1";
(10,20)*{\bullet}="2";
(20,10)*{\bullet}="5";
(40,0)*{\bullet}="8";
(10,40)*{\bullet}="4";
(10,30)*{\bullet}="3";
(13,30)*{\bullet}="7";
(16.5,20)*{\bullet}="6";
(32.5,10)*{\bullet}="9";
(25,20)*{\bullet}="10";
(17.5,30)*{\bullet}="11";
"1";"4" **\crv{} ?>* \dir{};
"2";"4" **\crv{} ?>* \dir{};
"5";"4" **\crv{} ?>* \dir{};
"8";"4" **\crv{} ?>* \dir{};
(50,43)*{y_{1}}="g";
(60,43)*{y_{2}}="g";
(70,43)*{y_{3}}="g";
(80,43)*{y_{4}}="g";
(70,0)*{\Gamma(\beta)}="g";
(85,20)*{\ldots}="h";
(50,30)*{\bullet}="12";
(50,40)*{\bullet}="13";
(60,20)*{\bullet}="14";
(60,30)*{\bullet}="15";
(60,40)*{\bullet}="16";
(70,10)*{\bullet}="17";
(70,20)*{\bullet}="18";
(70,30)*{\bullet}="19";
(70,40)*{\bullet}="20";
(80,0)*{\bullet}="21";
(80,10)*{\bullet}="22";
(80,20)*{\bullet}="23";
(80,30)*{\bullet}="24";
(80,40)*{\bullet}="25";
(85,20)*{\ldots}="111";
"12";"13" **\crv{} ?>* \dir{};
"14";"16" **\crv{} ?>* \dir{};
"17";"20" **\crv{} ?>* \dir{};
"21";"25" **\crv{} ?>* \dir{};
\endxy
\]
\caption{The digraphs of $\al$ and $\bt$ from Example~\ref{epcha}.}\label{f71}
\end{figure}

\begin{defi}\label{dsal}
{\rm
For $\al\in P(X)$, we define
\[
\sa=\sup\{\rho_\gam(x_0):\mbox{$\gam$ is a connected component of $\al$ of type $\cho$ with root $x_0$}\},
\]
where we agree that $\sa=0$ if $\al$ has no connected component of type $\cho$.
}
\end{defi}

Suppose $\al,\bt\in P(X)$ have a connected component of type $\cho$, but no cycles or rays.
Then, by Theorem~\ref{tpcha}, if $\al\!\con\!\bt$ then $\sa=\sbt$. However,
the converse is not true. Indeed, consider $\al,\bt\in P(X)$ from Example~\ref{epcha} (see Figure~\ref{f71}).
Then $\al$ is connected of type $\cho$ with the root of order $\omega$,
and $\bt$ is a join of connected components of type $\cho$ (chains) whose roots have orders $1,2,3,4,\ldots$.
Thus $\sa=\sbt=\omega$,
but $(\al,\bt)\notin\,\con$ by (5a) of Theorem~\ref{tpcha}. However, if $X$ is finite and $\al,\bt\in P(X)$ have no cycles,
then $\sa=\sbt$ does imply $\al\!\con\!\bt$.

The transformations of a finite $P(X)$ have no rays.
Hence, Theorem~\ref{tpcha} gives us the following corollary.

\begin{cor}\label{cpcha}
Let $X$ be finite, and let $\al,\bt\in P(X)$. Then $\al\!\con\!\bt$ if and only if $\cs(\al)=\cs(\bt)$ and $\sa=\sbt$.
\end{cor}

\begin{example}\label{ecpcha}
{\rm
Let $\al$ and $\bt$ be partial transformations whose digraphs are presented in Figures~\ref{f72} and~\ref{f73},
respectively. Then $\cs(\al)=\cs(\bt)=\{2,3\}$ and $\sa=\sbt=3$. Thus
 $\al\!\con\!\bt$ by Corollary~\ref{cpcha}.
}
\end{example}

\begin{figure}[h]
\[
\xy
(5,25)*{}*\cir<14pt>{};
(0,0)*{\bullet}="1";
(0,10)*{\bullet}="2";
(10,10)*{\bullet}="3";
(5,20)*{\bullet}="4";
(5,30)*{\bullet}="5";
"1";"2" **\crv{} ?>* \dir{};
"2";"4" **\crv{} ?>* \dir{};
"3";"4" **\crv{} ?>* \dir{};
(25,25)*{}*\cir<14pt>{};
(21.5,21.5)*{\bullet}="6";
(25,30)*{\bullet}="7";
(28,21.2)*{\bullet}="8";
(45,25)*{}*\cir<14pt>{};
(40,0)*{\bullet}="9";
(40,10)*{\bullet}="10";
(40,25)*{\bullet}="11";
(50,25)*{\bullet}="13";
(45,30)*{\bullet}="12";
(45,20)*{\bullet}="14";
(45,10)*{\bullet}="15";
(55,10)*{\bullet}="16";
"9";"10" **\crv{} ?>* \dir{};
"10";"11" **\crv{} ?>* \dir{};
"15";"14" **\crv{} ?>* \dir{};
"16";"14" **\crv{} ?>* \dir{};
(65,10)*{\bullet}="18";
(70,15)*{\bullet}="19";
(70,20)*{\bullet}="20";
(75,30)*{\bullet}="21";
(80,20)*{\bullet}="22";
(75,20)*{\bullet}="23";
(75,10)*{\bullet}="24";
"18";"19" **\crv{} ?>* \dir{};
"24";"19" **\crv{} ?>* \dir{};
"19";"20" **\crv{} ?>* \dir{};
"20";"21" **\crv{} ?>* \dir{};
"23";"21" **\crv{} ?>* \dir{};
"22";"21" **\crv{} ?>* \dir{};
\endxy
\]
\caption{The digraph of $\al$ from Example~\ref{ecpcha}.}\label{f72}
\end{figure}

\begin{figure}[h]
\[
\xy
(5,25)*{}*\cir<14pt>{};
(5,10)*{\bullet}="1";
(5,20)*{\bullet}="2";
(5,30)*{\bullet}="3";
"1";"2" **\crv{} ?>* \dir{};
(25,25)*{}*\cir<14pt>{};
(20,0)*{\bullet}="4";
(30,0)*{\bullet}="5";
(25,10)*{\bullet}="6";
(20,25)*{\bullet}="7";
(30,25)*{\bullet}="8";
(25,20)*{\bullet}="9";
"4";"6" **\crv{} ?>* \dir{};
"5";"6" **\crv{} ?>* \dir{};
"6";"9" **\crv{} ?>* \dir{};
(45,25)*{}*\cir<14pt>{};
(49.5,28)*{\bullet}="9";
(49.5,23)*{\bullet}="10";
(40.5,28)*{\bullet}="11";
(40.5,23)*{\bullet}="13";
(45,30)*{\bullet}="12";
(45,20)*{\bullet}="14";
(60,20)*{\bullet}="15";
(60,30)*{\bullet}="16";
"15";"16" **\crv{} ?>* \dir{};
(70,0)*{\bullet}="17";
(70,10)*{\bullet}="18";
(70,20)*{\bullet}="19";
(70,30)*{\bullet}="20";
(80,10)*{\bullet}="21";
"17";"18" **\crv{} ?>* \dir{};
"18";"19" **\crv{} ?>* \dir{};
"19";"20" **\crv{} ?>* \dir{};
"21";"19" **\crv{} ?>* \dir{};
\endxy
\]
\caption{The digraph of $\bt$ from Example~\ref{ecpcha}.}\label{f73}
\end{figure}

Using Theorem~\ref{tpcha}, we will count the conjugacy classes in $P(X)$ for an infinite set $X$ (Theorem~\ref{tpcc}).
We will use the aleph notation for the infinite cardinals, that is, for an ordinal $\vep$, we will write
$\aep$ for the cardinal indexed by $\vep$. If $\aep$ is viewed as an ordinal, we will consistently write $\oep$.
This is important because we will need to distinguish between ordinal and cardinal arithmetic. For example,
$\ome_0<\ome_0+1$ (ordinal arithmetic) but $\ale_0=\ale_0+1$ (cardinal arithmetic). It will be always clear
from the context which arithmetic is used.

A cardinal $\aep$ is called \emph{singular} if there is a limit ordinal $\vth<\oep$ and there is an
increasing transfinite sequence $\lan\lam_\nu\ran_{\nu<\vth}$ of ordinals $\lam_\nu<\oep$ such that
$\oep=\sup\{\lam_\nu:\nu<\vth\}$ \cite[page~160, Definition~2.1]{HrJe99}. (As in \cite{HrJe99},
``increasing'' means ``strictly increasing.'') If $\aep$ is not singular,
then it is called \emph{regular}.

For any cardinal $\aep$, the cardinal $\aepp$ is called the \emph{successor} cardinal of $\aep$.
Every successor cardinal is regular  \cite[page~162, Theorem~2.4]{HrJe99}. The following lemma
follows immediately from this fact and the definition of a regular cardinal.

\begin{lemma}\label{lst}
Let $\aepp$ be a successor cardinal and let $A$ be a set of ordinals such that $|A|<\aepp$
and $\lam<\oepp$ for every $\lam\in A$. Then $\sup\{\lam:\lam\in A\}<\oepp$.
\end{lemma}

To prove the counting theorem, we need a series of lemmas.

\begin{lemma}\label{lct1}
Let $|X|=\aep$ and let $\gam\in P(X)$ be of type $\cho$ with root $x_0$. Then $\rho(x_0)<\oepp$.
\end{lemma}
\begin{proof}
Let $x\in\spa(\gam)$.
We will prove that $\rho(x)<\oepp$ by well-founded induction.
If $x$ is $R_\gam$-minimal, then $\rho(x)=0<\oepp$.
Suppose $\rho(y)<\oepp$ for every $y\in x\gam^{-1}$. Then $\rho(y)+1<\oepp$ for every $y\in x\gam^{-1}$, and so
$\rho(x)=\sup\{\rho(y)+1:(y,x)\in R_\gam\}<\oepp$ by Lemma~\ref{lst}.
The result follows.
\end{proof}

\begin{lemma}\label{lct2}
Let $|X|=\aep$. Then for every nonzero ordinal $\mu<\oepp$, there is $\gam\in P(X)$
of type $\cho$ with root $x_0$ such that $\rho(x_0)=\mu$.
\end{lemma}
\begin{proof}
Let $0<\mu<\oepp$. We proceed by transfinite induction. The result is clearly true if $\mu=1$.
Let $\mu>1$ and suppose that the result is true for every ordinal $\lam$ such that $0<\lam<\mu$.

Fix $x_0\in X$, let $X_0=X\sm\{x_0\}$, and note that $|X_0|=\aep$. Since $\mu<\oepp$,
we have $|\mu|\leq\aep$. Thus, since $\aep\cdot\aep=\aep$
and $\mu=\{\lam:\mbox{$\lam$ is an ordinal such that $\lam<\mu$}\}$,
there is a collection $\{X_\lam\}_{0<\lam<\mu}$ of pairwise disjoint subsets of $X_0$
such that $|X_\lam|=\aep$ for every $\lam$.

Let $0<\lam<\mu$. By the inductive hypothesis,
there is $\gam_\lam\in P(X_\lam)$ of type $\cho$ with root $x_\lam$ such that $\rho(x_\lam)=\lam$.
We define $\gam\in P(X)$ as follows. Set
$\dom(\gam)=\bigcup_{0<\lam<\mu}\spa(\gam_\lam)$. For every $x\in\dom(\gam)$, define
\[
x\gam=\left\{\begin{array}{ll}
x\gam_\lam & \mbox{if $x\in\dom(\gam_\lam)$},\\
x_0 & \mbox{if $x=x_\lam$}.
\end{array}\right.
\]
Then $\gam$ is of type $\cho$, $x_0$ is the root of $\gam$, and $x_0\gam^{-1}=\{x_\lam:0<\lam<\mu\}$.
Let $\nu=\sup\{\rho(y):y\in x_0\gam^{-1}\}$. Then
\[
\nu=\sup\{\rho(x_\lam):0<\lam<\mu\}=\sup\{\lam:0<\lam<\mu\},
\]
where the last equality is true since $\rho(x_\lam)=\lam$ for every nonzero $\lam<\mu$.
Hence, either $\mu=\nu$ (if $\mu$ is a limit ordinal) or $\mu=\nu+1$ (if $\nu=\lam$ for some
nonzero $\lam<\mu$). It follows by (\ref{edord1}) that $\rho(x_0)=\mu$.
\end{proof}

\begin{lemma}\label{lct2a}
Let $|X|=\aep$ and let $\lan a_n\ran$ be an increasing sequence of ordinals $a_n<\oepp$
such that $a_0=0$. Then there is $\gam\in T(X)$ of type $\rro$ with a maximal right ray $\eta$
such that $\lan\eta^\gam_n\ran=\lan a_n\ran$.
\end{lemma}
\begin{proof}
Since $|X|=\aep$, there is a collection $\{X_n\}_{n\geq0}$ of pairwise disjoint subsets of $X$
such that $X_0=\{x_0\}$ and $|X_n|=\aep$ for every $n\geq1$. Let $n\geq1$. By Lemma~\ref{lct2},
there is $\gam_n\in P(X_n)$ of type $\cho$ with root $x_n$ such that $\rho_{\gam_n}(x_n)=a_n$.
Define $\gam\in T(X)$ by
\[
x\gam=\left\{\begin{array}{lll}
x\gam_n & \mbox{if $x\in\dom(\gam_n)$},\\
x_{n+1} & \mbox{if $x=x_n$},\\
x_1 & \mbox{for any other $x$}.
\end{array}\right.
\]
(See Figure~\ref{f61a}.)
Then $\gam$ is of type $\rro$ (since every $\gam_n$ is of type $\cho$). By the definition of $\gam$,
we have that $\rho_\gam(x_0)=0$ and $\eta=[x_0\,x_1\,x_2\ldots\ran$ is a maximal right ray in $\gam$. We have already
noticed that $\rho_\gam(x_0)=0=a_0$.
We will prove by induction on $n$ that $\rho_\gam(x_n)=a_n$ for every $n\geq1$.
Let $n=1$. Then, since $a_1\geq1$,
\[
\rho_\gam(x_1)=\max\{1,\rho_{\gam_1}(x_1)\}=\max\{1,a_1\}=a_1.
\]
Let $n\geq1$ and suppose $\rho_\gam(x_n)=a_n$. Then
\[
\rho_\gam(x_{n+1})=\max\{\rho_\gam(x_n)+1,\rho_{\gam_{n+1}}(x_{n+1})\}=\max\{a_n+1,a_{n+1}\}=a_{n+1},
\]
where the last equality is true since $\lan a_n\ran$ is increasing, and so $a_{n+1}>a_n$.
This concludes the inductive argument.
Thus $\eta^\gam_n=\rho_\gam(x_n)=a_n$ for every $n\geq0$, which completes the proof.
\end{proof}

\begin{figure}[h]
\[
\xy
(0,-3)*{x_0};
(0,0)*{\bullet}="x0";
(20,20)*{\bullet}="x1";
(40,40)*{\bullet}="x2";
(60,60)*{\bullet}="x3";
(64,64)*{\iddots};
(5,15)*{\bullet}="x";
(20,10)*{\bullet}="a";
(30,10)*{\bullet}="b";
(20,8)*{\vdots};
(30,8)*{\vdots};
(25,10)*{\ldots};
(58,62)*{x_3};
(38,42)*{x_2};
(18,22)*{x_1};
(3,17)*{x};
"x0";"x1" **\crv{} ?>* \dir{};
"x1";"x2" **\crv{} ?>* \dir{};
"x2";"x3" **\crv{} ?>* \dir{};
"a";"x1" **\crv{} ?>* \dir{};
"b";"x1" **\crv{} ?>* \dir{};
"x";"x1" **\crv{} ?>* \dir{};
(40,30)*{\bullet}="c";
(50,30)*{\bullet}="d";
(40,28)*{\vdots};
(50,28)*{\vdots};
(45,30)*{\ldots};
"c";"x2" **\crv{} ?>* \dir{};
"d";"x2" **\crv{} ?>* \dir{};
(60,50)*{\bullet}="e";
(70,50)*{\bullet}="f";
(60,48)*{\vdots};
(70,48)*{\vdots};
(65,50)*{\ldots};
"e";"x3" **\crv{} ?>* \dir{};
"f";"x3" **\crv{} ?>* \dir{};
\endxy
\]
\caption{The digraph of $\gam$ from Lemma~\ref{lct2a}.}\label{f61a}
\end{figure}

\begin{lemma}\label{lct3}
Let $\aepp$ be a successor cardinal. Then there is a collection $\{\lan a^\mu_n\ran\}_{\mu<\oepp}$
of increasing sequences $\lan a^\mu_n\ran$ of ordinals $a^\mu_n<\oepp$ such that for all ordinals $\mu,\lam<\oepp$,
$a^\mu_0=0$ and if $\lam<\mu$ then $a^\lam_m<a^\mu_n$ for all $m,n\geq1$.
\end{lemma}
\begin{proof}
We construct such a collection by transfinite recursion. We define $\lan a^0_n\ran=\lan0,1,2,3,\ldots\ran$.
Let $\mu$ be an ordinal such that $0<\mu<\oepp$ and suppose $\lan a^\lam_n\ran$ satisfying the hypotheses has been defined
for every ordinal $\lam<\mu$.
Let $A=\{a^\lam_n:\lam<\mu\mbox{ and }n\geq0\}$ and $\tau=\sup A$. Then $|A|=|\mu|\cdot\ale_0<\aepp$, and so
$\tau<\oepp$ by Lemma~\ref{lst}. Define $\lan a^\mu_n\ran=\lan 0,\tau+1,\tau+2,\tau+3,\ldots\ran$
and note that $\lan a^\mu_n\ran$ is an increasing sequence of ordinals $a^\mu_n<\oepp$ with $a^\mu_0=0$.
The construction has been completed.
It is clear from the construction that $a^\lam_m<a^\mu_n$ for all $\lam,\mu<\oepp$ with $\lam<\mu$ and all $m,n\geq1$.
\end{proof}

\begin{rem}\label{rct3}
{\rm
Let $\{\lan a^\mu_n\ran\}_{\mu<\oepp}$ be a collection from Lemma~\ref{lct3}.
Then it is clear that for all ordinals $\lam,\mu<\oepp$, if $\lam<\mu$ then $\lan a^\lam_n\ran$ does not dominate
$\lan a^\mu_n\ran$.
}
\end{rem}

\begin{defi}\label{dis}
{\rm
Let $\aepp$ be a successor cardinal. Denote by $\iso$ the set of all increasing
sequences $\lan a_n\ran$ of ordinals $a_n<\oepp$ such that $a_0=0$. Define a relation $\approx$ on $\iso$ by
\[
\lan a_n\ran\approx\lan b_n\ran\mbox { if $\lan b_n\ran$ dominates $\lan a_n\ran$ and $\lan a_n\ran$ dominates $\lan b_n\ran$}.
\]
It is straightforward to show that $\approx$ is an equivalence relation on $\iso$.
We denote by $[\lan a_n\ran]_\approx$ the equivalence class of $\lan a_n\ran$, and by $\iso^\approx$
the set of all equivalence classes of $\approx$.
}
\end{defi}

\begin{lemma}\label{lct4}
For any successor cardinal $\aepp$,
$|\iso|=\aepp^{\ale_0}$ and $\aepp\leq|\iso^\approx|\leq\aepp^{\ale_0}$.
\end{lemma}

\begin{proof}
Denote by $S_{\oepp}$ the set of all sequences $\lan s_n\ran$ of ordinals $s_n<\oepp$. Then
$S_{\oepp}$ is the set of all functions from $\mathbb N$ to $\oepp$, and so $|S_{\oepp}|=|\oepp|^{|\mathbb N|}=\aepp^{\ale_0}$.
Since $\iso$ is a subset of $S_{\oepp}$, we have $|\iso|\leq\aepp^{\ale_0}$.
Let $S^0_{\oepp}$ be the subset of $S_{\oepp}$ consisting of all sequences $\lan s_n\ran$
such that $s_n>0$ for all $n\geq0$. Then $|S^0_{\oepp}|=|S_{\oepp}|=\aepp^{\ale_0}$. Define a function
$f:S^0_{\oepp}\to\iso$ by $\lan s_n\ran f=\lan a_n\ran$, where
\[
a_0=0\mbox{ and }a_{n+1}=a_n+s_n\mbox{ for all $n\geq0$}.
\]
Then $f$ is injective (since for all ordinals $\mu,\lam_1,\lam_2$, if $\mu+\lam_1=\mu+\lam_2$ then $\lam_1=\lam_2$
\cite[page~120, Lemma~5.4]{HrJe99}), and so $|\iso|\geq|S^0_{\oepp}|=\aepp^{\ale_0}$.
We have proved that $|\iso|=\aepp^{\ale_0}$.

We have $|\iso^\approx|\leq|\iso|=\aepp^{\ale_0}$. Let $\{\lan a^\mu_n\ran\}_{\mu<\oepp}$ be a collection
of sequences constructed as in Lemma~\ref{lct3}. Then for all ordinals $\lam,\mu<\oepp$,
$\lan a^\mu_n\ran\in\iso$ and if $\lam<\mu$ then $\lan a^\lam_n\ran$ does not dominate $\lan a^\mu_n\ran$
(see Remark~\ref{rct3}).
It follows that any two different sequences from the collection
$\{\lan a^\mu_n\ran\}_{\mu<\oepp}$ are in different equivalence classes of $\approx$. Since
there are $\aepp$ sequences in the collection, it follows that $|\iso^\approx|\geq\aepp$.
This concludes the proof.
\end{proof}

We can now prove the counting theorem.
For a set $A$, we denote by $\mathcal{P}(A)$
the power set of $A$.

\begin{theorem}\label{tpcc}
Let $X$ be an infinite set with $|X|=\aep$. Then in $P(X)$ there are:
\begin{itemize}
  \item [\rm(1)] $\max\{2^{\ale_0},\ale_{\vep+1}\}$ conjugacy classes containing a representative with a cycle,
of which $\ale_0$ have a connected representative;
  \item [\rm(2)] $2^{\aep}$ conjugacy classes containing a representative with a connected component of type $\rro$,
but no cycles, of which at least $\aepp$ and at most $\aepp^{\ale_0}$ have a connected representative;
  \item [\rm(3)] $\ale_{\vep+1}$ conjugacy classes containing a representative with a connected component of type $\cho$, but no cycles
or connected components of type $\rro$,
of which $\ale_{\vep+1}$ have a connected representative.
\end{itemize}
In total,
there are $2^{\aep}$ conjugacy classes in $P(X)$, of which
at least $\aepp$ and at most $\aepp^{\ale_0}$ have a connected representative.
\end{theorem}

\begin{proof}
For $\al\in P(X)$,
we define $\ia,\ja\in\{0,1\}$ by $\ia=1$ if $\rho_\gam(x_0)=\sa$ for some
connected component of $\al$ of type $\cho$ with root $x_0$ (and $\ia=0$ otherwise);
and $\ja\in\{0,1\}$ by $\ja=1$ if $\al$ has a double ray (and $\ja=0$ otherwise).

To prove (1), let $A=\{[\al]_c:\mbox{$\al\in P(X)$ has a cycle}\}$.
Let
\[
A'=\{[\al]_c\in A:\mbox{$\al$ has no maximal left rays}\} \mbox{ and }
A''=\{[\al]_c\in A:\mbox{$\al$ has a maximal left ray}\}.
\]
By Theorem~\ref{tpcha}(4),
$\{A',A''\}$ is a partition of $A$.
Define $f':A'\to\mathcal{P}(\mathbb Z_+)\times(\oepp+1)\times\{0,1\}$ by
$([\al]_c)f'=(\cs(\al),\sa,\ia)$.
Then $f'$ is well defined and injective by Theorem~\ref{tpcha} and Lemma~\ref{lct1}.
(See Definition~\ref{dsal} and the discussion following the definition to see why $\ia$ is needed.)
Similarly, the mapping $f'':A''\to\mathcal{P}(\mathbb Z_+)\times(\oepp+1)\times\{0,1\}$
defined by $([\al]_c)f''=(\cs(\al),\sa,\ia)$ is well defined and injective. Thus
\[
|A'|\leq|\mathcal{P}(\mathbb Z_+)|\cdot|\oepp+1|\cdot2=2^{\ale_0}\cdot\aepp\cdot2=\max\{2^{\ale_0},\ale_{\vep+1}\},
\]
and the same holds for $|A''|$. Hence
\[
|A|=|A'|+|A''|=2|A'|\leq2\max\{2^{\ale_0},\ale_{\vep+1}\}=\max\{2^{\ale_0},\ale_{\vep+1}\}.
\]

Let $P$ be the set of prime positive integers. For any nonempty subset $Q\subseteq P$, let $\{\thet_q\}_{q\in Q}$
be a collection of completely disjoint cycles $\thet_q$ such that $\thet_q$ has length $q$ for every $q\in Q$.
(Such a collection exists since $X$ is infinite.)
Define $\bt_{\tq}\in P(X)$ by $\bt_{\tq}=\join_{q\in Q}\thet_q$.
For all nonempty subsets $Q_1,Q_2\subseteq P$ with $Q_1\ne Q_2$,
we have $(\bt_{\tq_1},\bt_{\tq_2})\notin\,\con$ by Theorem~\ref{tpcha}(1).
It follows that
$|A|\geq\mathcal{P}(P)=2^{\ale_0}$.
By Lemma~\ref{lct2}, for every nonzero ordinal $\mu<\oepp$, there is $\gam_\mu\in P(X)$
of type $\cho$ with root $x_0$ such that $\rho(x_0)=\mu$.
For all nonzero ordinals $\lam,\mu<\oepp$ with $\lam\ne\mu$, we have
$(\gam_\lam,\gam_\mu)\notin\,\con$ by Theorem~\ref{tpcha}(5).
It follows that
$|A|\geq|\oepp|=\aepp$. Hence $|A|\geq\max\{2^{\ale_0},\aepp\}$, and so
$|A|=\max\{2^{\ale_0},\aepp\}$.

Let $A_1=\{[\gam]_c:\mbox{$\gam\in P(X)$ has a cycle and $\gam$ is connected}\}$.
Fix a subset $X_0=\{x_0,x_1,\ldots\}$
of $X$, and for every integer $n\geq0$, define a cycle $\gam_n=(x_0\,x_1\ldots\,x_{n-1})\in P(X)$.
Then, by Proposition~\ref{pcom} and Theorem~\ref{tpcha},
$A_1=\{[\gam_0]_c,[\gam_1]_c,[\gam_2]_c,\ldots\}$, and so $|A_1|=\ale_0$. We have proved~(1).

To prove (2), let
\[
B=\{[\al]_c:\mbox{$\al\in P(X)$ has a connected component of type $\rro$, but no cycles}\},
\]
and let $B_1$ be the subset of $B$
consisting of all conjugacy classes $[\gam]_c\in B$ such that $\gam$ is connected.
Fix a double ray $\omega=\lan\ldots x_{-1}\,x_0\,x_1\ldots\ran\in P(X)$ and note that
\[
B_1=\{[\gam]_c:\mbox{$\gam\in P(X)$ is of type $\rro$}\}\cup\{[\omega]_c\}.
\]
Let $B'_1=\{[\gam]_c:\mbox{$\gam\in P(X)$ is of type $\rro$}\}$.
For every $\gam\in P(X)$ of type $\rro$,
we fix a maximal right ray $\eta^\gam$ in $\gam$. Define a function $g:B'_1\to\iso^\approx$ by
$([\gam]_c)g=[\lan\eta^\gam_n\ran]_\approx$. Note that $\lan\eta^\gam_n\ran\in\iso$ by Lemma~\ref{lct1}.
Suppose $[\gam_1]_c,[\gam_2]_c\in B'_1$ with $[\gam_1]_c=[\gam_2]_c$.
Then, by Theorem~\ref{tpcha}(3) and Lemma~\ref{lct5}, the sequences $\lan\eta^{\gam_1}_n\ran$
and $\lan\eta^{\gam_2}_n\ran$ dominate each other, and so
$[\lan\eta^{\gam_1}_n\ran]_\approx=[\lan\eta^{\gam_2}_n\ran]_\approx$. We have proved that $g$ is well defined.
The function $g$ is also injective (by Theorem~\ref{tpcha}(3)) and surjective (by Lemma~\ref{lct2a}).
Thus $|B'_1|=|\iso^\approx|$, and so, by Lemma~\ref{lct4}, $\aepp\leq|B'_1|\leq\aepp^{\ale_0}$.
Then $\aepp\leq|B_1|\leq\aepp^{\ale_0}$ since $|B_1|=|B'_1|+1$.

As to the cardinality of $B$, clearly $|B|\leq|P(X)|=(\aep+1)^{\aep}=2^{\aep}$.
Let
\begin{align}
B'&=\{[\al]_c\in B:\mbox{$\al$ has no maximal left rays or double rays}\},\notag\\
B''&=\{[\al]_c\in B:\mbox{$\al$ has a maximal left ray but no double rays}\}\notag.
\end{align}
By Theorem~\ref{tpcha}(3)(4),
$\{B',B'',\{[\omega]_c\}\}$ is a partition of $B$.

We will now prove that $|B'|\geq2^{\aep}$.
Since $|B'_1|\geq\aepp$, there is a collection $\{\gam_\mu\}_{\mu<\oepp}$ of transformations
$\gam_\mu\in P(X)$ of type $\rro$ such that
$(\gam_\mu,\gam_\lam)\notin\,\,\con$ if $\mu\ne\lam$.
Since $|\oep|=\aep$ and $\aep\cdot\aep=\aep$, there is a partition
$\{X_\mu\}_{\mu<\oep}$ of $X$ such that $|X_\mu|=|X|=\aep$ for every $\mu<\oep$. Let $\mu<\oep$. Since
$|X_\mu|=|X|$, there is a bijection $h_\mu:X_\mu\to X$. We can use $h_\mu$ to
obtain a ``copy'' of $\gam_\mu$ in $P(X_\mu)$:
define $\gam'_\mu\in P(X_\mu)$ by
\[
x\gam'_\mu=y\miff (xh_\mu)\gam_\mu=yh_\mu\mbox{ (for all $x,y\in X_\mu$)}.
\]
Let $\mu,\lam<\oep$ with $\mu\ne\lam$. Then $(\gam_\mu,\gam_\lam)\notin\,\,\con$, and so,
by Theorem~\ref{tpcha}(3) and Lemma~\ref{lct5}, $(\lan\eta_n\ran,\lan\xi_n\ran)\notin\,\,\approx$
for every maximal right ray $\eta$ in $\gam_\mu$ and every maximal right ray
$\xi$ in $\gam_\lam$. It follows that
\begin{equation}\label{e1tccc}
(\lan\eta'_n\ran,\lan\xi'_n\ran)\notin\,\,\approx
\end{equation}
for every maximal right ray $\eta'$ in $\gam'_\mu$ and every maximal right ray
$\xi'$ in $\gam'_\lam$. Let $K$ be a nonempty subset of $\oep$. Select $\nu=\nu_{\!\tk}\in K$ and a
maximal right ray $[x_0\,x_1\,x_2\ldots\ran$ in $\gam'_\nu$.
Define $\al_{\tk}\in P(X)$ by $\al_{\tk}=\join_{\mu\in K}\gam'_\mu$,
and note that $\al_{\tk}$ does not have a cycle or a double ray.
Let $K,L$ be nonempty subsets of $\oep$ such that $K\ne L$. We may assume that there is $\mu\in K$
such that $\mu\notin L$. Consider $\gam'_\mu$, which is a connected component of $\al_{\tk}$.
Let $\gam'_\lam$ be any connected component of $\al_{\tl}$. Then, by (\ref{e1tccc}),
$(\lan\eta'_n\ran,\lan\xi'_n\ran)\notin\,\,\approx$
for every maximal right ray $\eta'$ in $\gam'_\mu$ and every maximal right ray
$\xi'$ in $\gam'_\lam$. (Note that, by the definition of $\al_{\tk}$, this is also true when $\mu=\nu_{\!\tk}$
or $\lam=\nu_{\!\tl}$.) Thus $(\al_{\tk},\al_{\tl})\notin\,\,\con$ by Theorem~\ref{tcha}(3).
Hence any two different transformations from the collection
$\{\al_{\tk}\}_{\emptyset\ne K\subseteq\oep}$ are in different equivalence classes of $\con$. Since
there are $2^{\aep}$ transformations in the collection, it follows that $|B'|\geq2^{\aep}$.
Hence $|B|=|B'|+|B''|+|\{[\omega]_c\}|\geq|B'|\geq2^{\aep}$, and so $|B|=2^{\aep}$. We have proved~(2).

To prove (3),
let $C$ be the set of all $[\al]_c$ such that $\al\in P(X)$
has a connected component of type $\cho$, but no cycles or connected components of type $\rro$.
Let $C'=\{[\al]_c\in C:\mbox{$\al$ has no maximal left rays}\}$ and
$C''=\{[\al]_c\in C:\mbox{$\al$ has a maximal left ray}\}$. By Theorem~\ref{tpcha}(4),
$\{C',C''\}$ is a partition of $C$. Fix a maximal left ray $\lam=\lan\ldots x_2\,x_1\,x_0]\in P(X)$
and note that $C''=\{[\lam]_c\}$.
Define $h:C'\to(\oepp+1)\times\{0,1\}\times\{0,1\}$ by
$([\al]_c)h=(\sa,\ia,\ja)$.
Then $h$ is well defined and injective by Theorem~\ref{tpcha} and Lemma~\ref{lct1}, and so
$|C'|\leq\aepp\cdot2\cdot2=\ale_{\vep+1}$.
Thus $|C|=|C'|+|C''|=|C'|+1\leq\ale_{\vep+1}+1=\ale_{\vep+1}$.

Let $C_1$ be the subset of $C$
consisting of all $[\gam]_c\in C$ such that $\gam$ is connected. Note that
$C_1=\{[\gam]_c:\mbox{$\gam\in P(X)$ is of type $\cho$}\}\cup\{[\lam]_c\}$.
As in the proof of (1), we can construct a collection $\{\gam_\mu\}_{0<\mu<\oepp}$
of connected elements of $P(X)$ of type $\cho$ such that $(\gam_\lam,\gam_\mu)\notin\,\con$
if $\lam\ne\mu$. Thus $|C_1|\geq\aepp$, and so $\aepp\leq|C_1|\leq|C|\leq\aepp$.
Hence $|C|=|C_1|=\aepp$, which concludes the proof of (3).

The conjugacy classes considered in (1)--(3) cover all conjugacy classes in $P(X)$.
Thus, there are at most $\max\{2^{\ale_0},\ale_{\vep+1}\}+2^{\aep}+\aepp=2^{\aep}$
conjugacy classes in $P(X)$
(which also follows from the fact that $|P(X)|=2^{\aep}$).
By (2), there are at least $2^{\aep}$ conjugacy classes, so the number of conjugacy classes
in $P(X)$ is $2^{\aep}$.
By (1)--(3), at least $\aepp$ and at most $\ale_0+\aepp^{\ale_0}+\aepp=\aepp^{\ale_0}$ of these conjugacy classes have
a connected representative.
(We point out that if a conjugacy class has a connected representative,
it does not imply that all representatives of this class are connected.)
\end{proof}

\section{Conjugacy in $T(X)$}\label{stx}
\setcounter{equation}{0}
\setcounter{figure}{0}
A characterization of the conjugacy $\con$ in the monoid $T(X)$ of full transformations on $X$
is simpler than that of the conjugacy in $P(X)$ (see Section~\ref{spx}). The reason is
that a connected component of $\al\in T(X)$ cannot have a maximal left ray or a maximal chain.
Suppose $\al,\bt\in T(X)$ and $\al\con\bt$ in $P(X)$. Then $\al\phi=\phi\bt$ and $\bt\psi=\psi\al$ for some
rp-homomorphisms $\phi$ and $\psi$. By Lemma~\ref{lgpa}, $X=\spa(\al)\subseteq\dom(\phi)$ and
$X=\spa(\bt)\subseteq\dom(\psi)$. Therefore, $\phi,\psi\in T(X)$, and so $\al\con\bt$ in $T(X)$.
In other words, $\con$ in $T(X)$ is the restriction of $\con$ in $P(X)$ to $T(X)\times T(X)$.

These observations and Theorem~\ref{tpcha} give a characterization of $\con$ in $T(X)$.

\begin{theorem}\label{tcha}
Let $\al,\bt\in T(X)$. Then $\al\!\con\!\bt$ in $T(X)$ if and only if exactly one of the following conditions is satisfied:
\begin{itemize}
  \item [\rm(1)] both $\al$ and $\bt$ have a cycle and $\cs(\al)=\cs(\bt)$;
  \item [\rm(2)] both $\al$ and $\bt$ have a double ray but no cycles;
  \item [\rm(3)] all connected components of both $\al$ and $\bt$ have type $\rro$ and:
\begin{itemize}
  \item [\rm(a)] for every connected component $\gam$ of $\al$, there is
a connected component $\del$ of $\bt$ such that $\lan \xi^\del_n\ran$ dominates $\lan \eta^\gam_n\ran$
for some maximal right ray $\eta$ in $\gam$ and some maximal right ray $\xi$ in $\del$, and
  \item [\rm(b)] for every connected component $\del$ of $\bt$, there is
a connected component $\gam$ of $\al$ such that $\lan \eta^\gam_n\ran$ dominates $\lan \xi^\del_n\ran$
for some maximal right ray $\xi$ in $\del$ and some maximal right ray $\eta$ in $\gam$.
\end{itemize}
\end{itemize}
\end{theorem}

\begin{example}\label{echa}
{\rm
Let $X=\{x_0,x_1,x_2,\ldots,y_1,y_2,y_3,\ldots\}$ and consider
\begin{align}
\al&=[x_0\,y_0\,x_1\,y_1\,x_2\,y_2\ldots\ran,\notag\\
\bt&=[x_0\,y_0\,x_1\,y_1\,x_2\,y_2\ldots\ran\jo[y_1\,y_2\,x_1]\jo[y_3\,y_4\,y_5\,y_6\,x_2]\jo[y_7\,y_8\,y_9\,y_{10}\,y_{11}\,y_{12}\,x_3]\jo\cdots\notag
\end{align}
in $T(X)$ (see Figure~\ref{f61}).
We will argue that $\al$ and $\bt$ are not conjugate.
Both $\al$ and $\bt$ are connected of type $\rro$. The only maximal right ray in $\al$
is $\eta=[x_0\,y_0\,x_1\,y_1\,x_2\,y_2\ldots\ran$ with $\lan \eta^\gam_n\ran=\lan n\ran$ (where $\gam=\al$). If $\al$ and $\bt$
were conjugate, then $\lan \eta_n\ran$ would dominate $\lan \xi^\del_n\ran$ (where $\del=\bt$)
for some maximal right ray $\xi$ in $\bt$, and so for all
maximal right rays $\xi$ in $\bt$ (see Lemma~\ref{lct5}). The right chain $\xi=[x_0\,y_0\,x_1\,y_1\,x_2\,y_2\ldots\ran$
is a maximal right chain in $\bt$ with $\lan \xi^\del_n\ran=\lan 2n\ran$. It is clear that the sequence $\lan n\ran$
does not dominate the sequence $\lan 2n\ran$. Hence, by Theorem~\ref{tcha}, $\al$ and $\bt$ are not conjugate.
}
\end{example}

\begin{figure}[h]
\[
\xy
(0,-5)*{\Gamma(\alpha)}="g";
(0,0)*{\bullet}="1";
(0,10)*{\bullet}="2";
(0,20)*{\bullet}="3";
(0,30)*{\bullet}="4";
(0,40)*{\bullet}="5";
(0,45)*{\vdots}="a";
"1";"2" **\crv{} ?>* \dir{};
"3";"2" **\crv{} ?>* \dir{};
"3";"4" **\crv{} ?>* \dir{};
"4";"5" **\crv{} ?>* \dir{};
(40,-5)*{\Gamma(\beta)}="g";
(20,0)*{\bullet}="6";
(20,10)*{\bullet}="7";
(20,20)*{\bullet}="8";
(20,30)*{\bullet}="9";
(20,40)*{\bullet}="10";
(20,45)*{\vdots}="a";
"6";"7" **\crv{} ?>* \dir{};
"8";"7" **\crv{} ?>* \dir{};
"8";"9" **\crv{} ?>* \dir{};
"9";"10" **\crv{} ?>* \dir{};
(30,0)*{\bullet}="11";
(40,0)*{\bullet}="12";
(50,0)*{\bullet}="13";
(60,0)*{\bullet}="14";
(25,5)*{\bullet}="11a";
(35,5)*{\bullet}="12a";
(45,5)*{\bullet}="13a";
(55,5)*{\bullet}="14a";
(30,10)*{\bullet}="12b";
(40,10)*{\bullet}="13b";
(50,10)*{\bullet}="14b";
(25,15)*{\bullet}="12c";
(35,15)*{\bullet}="13c";
(45,15)*{\bullet}="14c";
(30,20)*{\bullet}="13d";
(40,20)*{\bullet}="14d";
(25,25)*{\bullet}="13e";
(35,25)*{\bullet}="14e";
(30,30)*{\bullet}="14f";
(25,35)*{\bullet}="14g";
"6";"7" **\crv{} ?>* \dir{};
"7";"8" **\crv{} ?>* \dir{};
"8";"9" **\crv{} ?>* \dir{};
"9";"10" **\crv{} ?>* \dir{};
"11";"11a" **\crv{} ?>* \dir{};
"11a";"7" **\crv{} ?>* \dir{};
"12";"8" **\crv{} ?>* \dir{};
"13";"9" **\crv{} ?>* \dir{};
"14";"10" **\crv{} ?>* \dir{};
\endxy
\]
\caption{The digraphs of $\al$ and $\bt$ from Example~\ref{echa}.}\label{f61}
\end{figure}

If $X$ is a finite set, then every $\al\in T(X)$ has a cycle. Hence, Theorem~\ref{tcha} gives us the following corollary.

\begin{cor}\label{ccha}
Let $X$ be finite, and let $\al,\bt\in T(X)$. Then $\al\!\con\!\bt$ if and only if $\cs(\al)=\cs(\bt)$.
\end{cor}

Modifying the proof of Theorem~\ref{tpcc}, we can count the number of conjugacy classes in an infinite $T(X)$.

\begin{theorem}\label{ttcc}
Let $X$ be an infinite set with $|X|=\aep$. Then in $T(X)$ there are:
\begin{itemize}
  \item [\rm(1)] $2^{\ale_0}$ conjugacy classes consisting of transformations with a cycle, of which $\ale_0$
have a connected representative;
  \item [\rm(2)] one conjugacy class consisting of transformations with a double ray but not a cycle;
  \item [\rm(3)] $2^{\aep}$ conjugacy classes consisting of transformations without a cycle or a double ray, of which
at least $\aepp$ and at most $\aepp^{\ale_0}$ have a connected representative.
\end{itemize}
In total,
there are $2^{\aep}$ conjugacy classes in $T(X)$, of which
at least $\aepp$ and at most $\aepp^{\ale_0}$ have a connected representative.
\end{theorem}

The reason for (1) is that $\al\in T(X)$ does not have any maximal left rays or components of type $\cho$.
Thus, the set $A=A'\cup A''$ from the proof of (1) of Theorem~\ref{tpcc} reduces to $A'$,
and the function $f':A'\to\mathcal{P}(\mathbb Z_+)\times(\oepp+1)\times\{0,1\}$ reduces to $f':A'\to\mathcal{P}(\mathbb Z_+)$.
The reason for (2) is that if $\al\in T(X)$ has a double ray but not a cycle, then each component of $\al$
either has a double ray or is of type $\rro$. Any two such transformations are then conjugate by Lemma~\ref{ldch}
and Proposition~\ref{pspl}.

\section{Conjugacy in $\gx$}\label{sgx}
\setcounter{equation}{0}
\setcounter{figure}{0}

By $\gx$ we mean the subsemigroup of $T(X)$ consisting of injective transformations.
If $X$ is finite, then $\gx=\sym(X)$ but this is not the case for an infinite $X$.
The semigroup $\gx$ is universal for right cancellative semigroups with no idempotents
(except possibly the identity): that is, any such semigroup can be embedded in $\gx$ for some $X$ \cite[Lemma~1.0]{ClPr64}.
The semigroup $\gx$ has been studied mainly in the context of:
ideals and congruences \cite{LeSc89,Sut61};
$\mathcal{G}(X)$-normal semigroups \cite{Le94,Le01,Su80};
Baer-Levi semigroups \cite{LeSu83,LiMa76}; $\mathit{BQ}$-semigroups \cite{Ke02,Su09},
and centralizers \cite{Ko10,Ko11}.
In this section, we characterize the conjugacy $\con$ in $\gx$ for an arbitrary set $X$.

We note that every connected transformation in $P(X)$ that is also injective is a cycle, a ray, or
a chain. Since transformations in $\gx$ are full, $\al\in\gx$ cannot contain a maximal left ray
or a maximal chain. These observations give the following lemma.

\begin{lemma}\label{lgxcon}
Let $\al\in\gx$. Then every connected component of $\al$ is a right ray, a double ray, or a cycle.
\end{lemma}

The following proposition follows from Lemma~\ref{lgxcon} and Proposition~\ref{pdec}.

\begin{prop}\label{pgxuni}
Let $\al\in\gx$. Then there exist unique sets: $A$ of right rays,
$B$ of double rays, and $C$ of cycles such that
the transformations in $A\cup B\cup C$ are pairwise completely disjoint and
\begin{equation*}
\al=\left(\join_{\eta\in A}\eta\right)\jo\left(\join_{\omega\in B}\omega\right)\jo\left(\join_{\thet\in C}\thet\right).\label{euni}
\end{equation*}
\end{prop}

Let $\al\in\gx$.
We will denote the unique sets $A$, $B$, and $C$ from Proposition~\ref{pgxuni} by $\aal$, $\ba$, and $\ca$, respectively.
For $n\geq1$, we will denote by $\cna$ the subset of $\ca$ consisting of cycles of length $n$. Note that:
\begin{align}
\aal&=\mbox{the set of maximal right rays contained in $\al$},\notag\\
\ba&=\mbox{the set of double rays contained in $\al$},\notag\\
\ca&=\mbox{the set of cycles contained in $\al$}.\notag
\end{align}

For $\eta=[x_0\, x_1\, x_2\ldots\ran$, $\omega=\lan\ldots x_{-1}\, x_0\, x_1\ldots\ran$,
$\thet=(x_0\, x_1\ldots x_{k-1})$, and any $\phi$ in $\gx$, we define:
\[
\eta\phi^*=[x_0\phi\,\, x_1\phi\,\, x_2\phi\,\ldots\ran,\,\,
\omega\phi^*=\lan\ldots x_{-1}\phi\,\, x_0\phi\,\, x_1\phi\ldots\ran,\,\,
\thet\phi^*=(x_0\phi\,\, x_1\phi\ldots x_{k-1}\phi).
\]

\begin{prop}\label{pgxhom}
Let $\al,\bt,\phi\in\gx$. Then $\phi$ is a homomorphism from $\Gamma(\al)$ to $\Gamma(\bt)$ if and only if for all $\eta\in\aal$,
$\omega\in\ba$, and $\thet\in\ca$:
\begin{itemize}
  \item[\rm(1)] either there is a unique $\eta_1\in\ab$ such that $\eta\phi^*\sqs\eta_1$
or there is a unique $\omega_1\in\bb$ such that $\eta\phi^*\sqs\omega_1$;
  \item[\rm(2)] $\omega\phi^*\in\bb$ and $\thet\phi^*\in\cb$.
\end{itemize}
\end{prop}
\begin{proof} Suppose $\phi$ is a homomorphism from $\Gamma(\al)$ to $\Gamma(\bt)$.
Let $\eta=[x_0\,x_1\,x_2\ldots\ran\in\aal$. Then, since $\phi$ is an injective homomorphism,
$\eta\phi^*=[x_0\phi\,\,x_1\phi\,\,x_2\phi\,\ldots\ran$ is a right ray in $\Gamma(\bt)$.
By the proof of Proposition~\ref{pspl}, $\phi|_{\spa(\eta)}$ is a homomorphism from $\Gamma(\eta)$
to $\Gamma(\gam)$ for some connected component $\gam$ of $\bt$. By Lemma~\ref{lgxcon}, either $\gam=\eta_1=(y_0\,y_1\,y_2\ldots\ran$
is a right ray in $\bt$ or $\gam=\omega_1=\lan\ldots y_{-1}\,y_0\,y_1\ldots\ran$ is a double ray in $\bt$ ($\gam$ cannot be a cycle since
$\phi$ is injective). In the former case, $\eta\phi^*\sqs\eta_1$,
and in the latter case, $\eta\phi^*\sqs\omega_1$. The uniqueness of $\eta_1$ and $\omega_1$
follows from the fact that the elements of $\ab\cup\bb$ are pairwise completely disjoint.
We have proved (1). The proof of (2) is similar.

Conversely, suppose that $\phi$ satisfies (1) and (2). Then it follows immediately that
for all $x,y\in X$, $x\ara y$ implies $x\phi\arb y\phi$, and so $\phi$ is a homomorphism from $\Gamma(\al)$ to $\Gamma(\bt)$.
\end{proof}

\begin{defi}\label{dgxh}
{\rm
Let $\al,\bt\in\gx$. For a homomorphism $\phi\in\gx$ from $\Gamma(\al)$ to $\Gamma(\bt)$,
we define a mapping $\hph:\aal\cup\ba\cup\ca\to\ab\cup\bb\cup\cb$ by:
\[
\delta\hph=
\left\{\begin{array}{lll}
\eta & \mbox{if $\delta\in\aal$ and $\delta\phi^*\sqs\eta$ for some $\eta\in\ab$,}\\
\omega & \mbox{if $\delta\in\aal$ and $\delta\phi^*\sqs\omega$ for some $\omega\in\bb$,}\\
\delta\phi^* & \mbox{if $\delta\in\ba\cup\ca$.}
\end{array}\right.
\]
Note that $\hph$ is well defined (by Proposition~\ref{pgxhom}) and injective (since $\phi$ is injective).
}
\end{defi}

We will need the following lemma from set theory (whose proof is straightforward).
\begin{lemma}\label{lgxset}
Let $A_1$, $B_1$, $A_2$, and $B_2$ be sets such that
$A_1\cap B_1=\emptyset$, $A_2\cap B_2=\emptyset$, $|A_1|+|B_1|\leq|A_2|+|B_2|$, and $|B_1|\leq|B_2|$.
Then there is an injective mapping $f:A_1\cup B_1\to A_2\cup B_2$ such that
$xf\in B_2$ for every $x\in B_1$.
\end{lemma}

We can now characterize the conjugacy $\con$ in $\gx$.

\begin{theorem}\label{tgxcha}
Let $\al,\bt\in\gx$. Then $\al\!\con\!\bt$ in $\gx$ if and only if
$|\aal|+|\ba|=|\ab|+|\bb|$, $|\ba|=|\bb|$, and $|\cna|=|\cnb|$ for every $n\geq1$.
\end{theorem}
\begin{proof}
Suppose $\al\!\con\!\bt$ in $\gx$. Then, by Corollary~\ref{cconh}, there is
$\phi\in\gx$ such that $\phi$ is a homomorphism from $\Gamma(\al)$ to $\Gamma(\bt)$.
Define $f:\aal\cup\ba\to\ab\cup\bb$ by $\del f=\del\hph$. (By the definitions of $\hph$ and $\phi^*$,
$\del f$ is indeed in $\ab\cup\bb$ if $\del\in\aal\cup\ba$.) The mapping $f$ is injective (since $\hph$ is injective),
$\aal\cap\ba=\emptyset$, and $\ab\cap\bb=\emptyset$. Thus
\[
|\aal|+|\ba|=|\aal\cup\ba|\leq|\ab\cup\bb|=|\ab|+|\bb|.
\]
Similarly, $|\ba|\leq|\bb|$
since $g:\ba\to\bb$ defined by $\omega g=\omega\hph$ is well defined and injective. Let $n\geq1$.
Define $h:\cna\to\cnb$ by $\thet h=\thet\hph$.
(If $\thet=(x_0\ldots\,x_{n-1})\in\cna$, then $\thet\hph=\thet\phi^*=(x_0\phi\ldots\,x_{n-1}\phi)\in\cnb$.)
The mapping $h$ is injective, and so $|\cna|\leq|\cnb|$. By symmetry, $|\ab|+|\bb|\leq|\aal|+|\ba|$, $|\bb|\leq|\ba|$,
and $|\cnb|\leq|\cna|$. Hence the stated equalities hold.

Conversely, suppose $|\aal|+|\ba|=|\ab|+|\bb|$, $|\ba|=|\bb|$, and $|\cna|=\cnb|$ for every $n\geq1$.
We will define an injective homomorphism $\phi$ from $\Gamma(\al)$ to $\Gamma(\bt)$. By Lemma~\ref{lgxset},
there is an injective mapping $f:\aal\cup\ba\to\ab\cup\bb$ such that $\omega f\in\bb$ for every
$\omega\in\ba$. For every $n\geq1$, fix a bijection $g_n:\cna\to\cnb$. Let $n\geq1$. For all $\eta\in\aal$, $\omega\in\ba$,
and $\thet\in\cna$, we define $\phi$ on $\dom(\eta)\cup\dom(\omega)\cup\dom(\thet)$
in such a way that $\eta\phi^*\sqs\eta f$, $\omega\phi^*=\omega f$, and $\thet\phi^*=\thet g_n$.
Note that this defines $\phi$ for every $x\in X$. By the definition of $\phi$ and Proposition~\ref{pgxhom},
$\phi\in\gx$ and $\phi$ is a homomorphism from $\Gamma(\al)$ to $\Gamma(\bt)$.
By symmetry, there is an injective homomorphism $\psi$ from $\Gamma(\bt)$ to $\Gamma(\al)$.
Hence $\al\!\con\!\bt$ by Corollary~\ref{cconh}.
\end{proof}

\begin{example}\label{egxcha}
{\rm
Let $\al$ and $\bt$ be partial transformations on an infinite set whose digraphs are presented in Figures~\ref{f81} and~\ref{f82},
respectively. Then $|\aal|=1$, $|\ba|=\ale_0$, $|\aal|+|\ba|=\ale_0$, and $|\cna|=0$ for every $n\geq1$.
Also, $|\ab|=2$, $|\bb|=\ale_0$,
$|\ab|+|\bb|=\ale_0$, and $|\cnb|=0$ for every $n\geq1$.
Thus $\al\!\con\!\bt$ by Theorem~\ref{tgxcha}.
}
\end{example}

\begin{figure}[h]
\[
\xy
(0,0)*{\bullet}="1";
(0,10)*{\bullet}="2";
(0,20)*{\bullet}="3";
(0,30)*{\bullet}="4";
(0,35)*{\vdots}="b";
"1";"2" **\crv{} ?>* \dir{};
"2";"3" **\crv{} ?>* \dir{};
"3";"4" **\crv{} ?>* \dir{};
(10,-3)*{\vdots}="a";
(10,0)*{\bullet}="5";
(10,10)*{\bullet}="6";
(10,20)*{\bullet}="7";
(10,30)*{\bullet}="8";
(10,35)*{\vdots}="c";
"5";"6" **\crv{} ?>* \dir{};
"6";"7" **\crv{} ?>* \dir{};
"7";"8" **\crv{} ?>* \dir{};
(20,-3)*{\vdots}="a";
(20,0)*{\bullet}="9";
(20,10)*{\bullet}="10";
(20,20)*{\bullet}="11";
(20,30)*{\bullet}="12";
(20,35)*{\vdots}="c";
"9";"10" **\crv{} ?>* \dir{};
"10";"11" **\crv{} ?>* \dir{};
"11";"12" **\crv{} ?>* \dir{};
(30,-3)*{\vdots}="a";
(30,0)*{\bullet}="13";
(30,10)*{\bullet}="14";
(30,20)*{\bullet}="15";
(30,30)*{\bullet}="16";
(30,35)*{\vdots}="c";
(35,20)*{\cdots}="c";
"13";"14" **\crv{} ?>* \dir{};
"14";"15" **\crv{} ?>* \dir{};
"15";"16" **\crv{} ?>* \dir{};
\endxy
\]
\caption{The digraph of $\al$ from Example~\ref{egxcha}.}\label{f81}
\end{figure}

\begin{figure}[h]
\[
\xy
(-10,0)*{\bullet}="01";
(-10,10)*{\bullet}="02";
(-10,20)*{\bullet}="03";
(-10,30)*{\bullet}="04";
(-10,35)*{\vdots}="b";
"01";"02" **\crv{} ?>* \dir{};
"02";"03" **\crv{} ?>* \dir{};
"03";"04" **\crv{} ?>* \dir{};
(0,0)*{\bullet}="1";
(0,10)*{\bullet}="2";
(0,20)*{\bullet}="3";
(0,30)*{\bullet}="4";
(0,35)*{\vdots}="b";
"1";"2" **\crv{} ?>* \dir{};
"2";"3" **\crv{} ?>* \dir{};
"3";"4" **\crv{} ?>* \dir{};
(10,-3)*{\vdots}="a";
(10,0)*{\bullet}="5";
(10,10)*{\bullet}="6";
(10,20)*{\bullet}="7";
(10,30)*{\bullet}="8";
(10,35)*{\vdots}="c";
"5";"6" **\crv{} ?>* \dir{};
"6";"7" **\crv{} ?>* \dir{};
"7";"8" **\crv{} ?>* \dir{};
(20,-3)*{\vdots}="a";
(20,0)*{\bullet}="9";
(20,10)*{\bullet}="10";
(20,20)*{\bullet}="11";
(20,30)*{\bullet}="12";
(20,35)*{\vdots}="c";
"9";"10" **\crv{} ?>* \dir{};
"10";"11" **\crv{} ?>* \dir{};
"11";"12" **\crv{} ?>* \dir{};
(30,-3)*{\vdots}="a";
(30,0)*{\bullet}="13";
(30,10)*{\bullet}="14";
(30,20)*{\bullet}="15";
(30,30)*{\bullet}="16";
(30,35)*{\vdots}="c";
(35,20)*{\cdots}="c";
"13";"14" **\crv{} ?>* \dir{};
"14";"15" **\crv{} ?>* \dir{};
"15";"16" **\crv{} ?>* \dir{};
\endxy
\]
\caption{The digraph of $\bt$ from Example~\ref{egxcha}.}\label{f82}
\end{figure}

Using Theorem~\ref{tgxcha}, we can count the conjugacy classes in $\gx$. First, we need the following lemma.

\begin{lemma}\label{lgxaal}
Let $X$ be an infinite set with $|X|=\aep$, let $\al\in\gx$.
Then $|\aal|\leq\aep$, $|\ba|\leq\aep$, and $|\cna|\leq\aep$ for every $n\geq1$.
\end{lemma}
\begin{proof}
Let $Y=\bigcup_{\eta\in\aal}\dom(\eta)\subseteq X$. Since the elements of $\aal$ are pairwise completely disjoint
and $|\dom(\eta)|=\ale_0$ for every $\eta\in\aal$, we have
\[
\aep=|X|\geq|Y|=|\bigcup_{\eta\in\aal}\dom(\eta)|=|\aal|\cdot\ale_0\geq|\aal|.
\]
Thus $|\aal|\leq\aep$. The proofs for $\ba$ and $\cna$ ($n\geq1$) are similar.
\end{proof}

For sets $A$ and $B$, we denote by $A^B$ the set of all functions from $B$ to $A$.

\begin{theorem}\label{tgxc}
Let $X$ be an infinite set with $|X|=\aep$. Let $\kappa=\ale_0+|\vep|$. Then there are
$\kappa^{\ale_0}$ conjugacy classes in $\gx$, of which two have a connected representative
if $\aep=\ale_0$, and none has a connected representative if $\aep>\ale_0$.
\end{theorem}
\begin{proof}
Let $K$ be the set of all cardinals $\tau$ such that $\tau\leq\aep$. Then $K$ contains $\ale_0$
finite cardinals and $|\vep|+1$ infinite cardinals, hence $|K|=\ale_0+|\vep|+1=\ale_0+|\vep|=\kappa$.
Let $\gx/\!\!\con$ be the set of conjugacy classes of $\gx$.
Define a function $f:\gx/\!\!\con\,\to K^{\mathbb N}$, where $\mathbb N=\{0,1,2,\ldots\}$, by
\[
([\al]_c)f=(|\aal|+|\ba|,|\ba|,|C^1_\al|,|C^2_\al|,|C^3_\al|,\ldots).
\]
By Theorem~\ref{tgxcha},
$f$ is well defined and injective. Thus $|\gx/\!\!\con\!\!|\leq|K^{\mathbb N}|=|K|^{|\mathbb N|}=\kappa^{\ale_0}$.

We next define an injective
mapping $g:K^{\mathbb N}\to\gx/\!\!\con$. Let
\[
\xi=(\tau_2,\tau_3,\tau_4,\ldots)\in K^{\mathbb N}.
\]
(It will be clear from the definition of $g$ why we begin the indexing with $n=2$.) Let $\tau=\sum_{n=2}^\infty n\tau_n$
(see \cite[Chapter~9]{HrJe99}). For every $n\geq2$, $n\tau_n\leq\aep$ (since $\tau_n\leq\aep$ and $\aep$ is infinite). Thus
\[
\tau=\sum_{n=2}^\infty n\tau_n\leq\ale_0\cdot\aep=\aep,
\]
and so $\aep+\tau=\aep$. Hence, there is a collection $\{X_n\}_{n\geq1}$ of pairwise disjoint subsets of $X$ such that
$\bigcup_{n=1}^\infty X_n=X$,
$|X_1|=\aep$, and $|X_n|=n\tau_n$ for every $n\geq2$. Let $n\geq2$. Since $|X_n|=n\tau_n$, there is a collection
$C_n$ of $n$-cycles in $\gx$ such that $|C_n|=\tau_n$ and $\dom(\join_{\,\thet\in C_n}\!\thet)=X_n$. Let $\al_n=\join_{\,\thet\in C_n}\!\thet$.
Define a transformation $\al_\xi$ on $X$ by
\[
\al_\xi=\join_{n\geq2}\!\al_n\jo\join_{x\in X_1}\!(x).
\]
Then $\al\in\gx$, $\aal=\ba=\emptyset$, and $\cna=C_n$ for all $n\geq2$.
Thus
\[
(|C^1_\al|,|C^2_\al|,|C^3_\al|,|C^4_\al|,\ldots)=(\aep,\tau_2,\tau_3,\tau_4,\ldots),
\]
and it follows from Theorem~\ref{tgxcha} that the mapping $g:K^{\mathbb N}\to\gx/\!\!\con$ defined by
$\xi g=\al_\xi$ is injective. Hence
$|\gx/\!\!\con\!\!|\geq|K^{\mathbb N}|=|K|^{|\mathbb N|}=\kappa^{\ale_0}$.

Suppose $|X|=\ale_0$, say $X=\{x_1,x_2,x_3,\ldots\}$.
Then, by Theorem~\ref{tgxcha} and Lemma~\ref{lgxcon}, the only conjugacy classes in $\gx$
with a connected representative are $[(x_1\,x_2\,x_3\ldots\ran]$ and $[\lan\ldots\,x_6\,x_4\,x_2\,x_1\,x_3\,x_5\ldots\ran]$.
(There is no single cycle in $\thet$ in $\gx$ since $\dom(\thet)$ is finite.)

If $|X|>\ale_0$, then no element $\al\in\gx$ is connected since $\dom(\al)=X$
and the domain of any right ray, double ray, or cycle has cardinality at most $\ale_0$.
The result follows.
\end{proof}

\section{Problems}\label{spro}

The results of this paper prompt a number of problems in combinatorics, semigroups, matrix theory, and set theory.
The first problem asks for the number of conjugacy classes in some important finite semigroups.
\begin{prob}
{\rm
Let $X$ be a finite set. Is it possible to find a closed formula that gives the number of conjugacy classes in $T(X)$, $P(X)$ or $\mi(X)$
(where $\mi(X)$ denotes the symmetric inverse semigroup on $X$)?
}
\end{prob}

The second problem might attract the attention of experts in set theory.
\begin{prob}
{\rm
Let $X$ be an infinite set with $|X|=\aep$. According to Theorem \ref{ttcc}, the number
of conjugacy classes in $T(X)$ that have a connected representative is in the interval
$[\aepp,\aepp^{\ale_0}]$.
Is it possible to be more precise and reduce the length of this interval?
}
\end{prob}

In this paper we characterized the conjugate elements in some well-known transformation semigroups, but
there are many other transformation semigroups, or endomorphism monoids of some relational algebras
that may be considered.
\begin{prob}
{\rm
Characterize $\con$, and calculate the number of conjugacy classes,  in other transformation semigroups such as, for example, those appearing in the problem list of \cite[Section~6]{ak} or those
appearing in the large list of transformation semigroups included in \cite{vhf}.  Especially interesting would be a characterization of the conjugacy classes in the centralizers of idempotents \cite{arko2,arko1}.
}
\end{prob}

The theorems and problems in this paper have natural linear counter-parts.
\begin{prob}
{\rm
Characterize $\con$ in the endomorphism monoid of a (finite or infinite dimensional) vector space.
}
\end{prob}

Whenever some result holds for both sets and vector spaces the natural step forward is to prove those results for independence algebras.
\begin{prob}
{\rm
Characterize $\con$ in the endomorphism monoid of a (finite or infinite dimensional) independence algebra.
(For historical notes on the importance of these algebras, see \cite{aeg,arfo}; for definitions and basic results, see
\cite{araujo2,araujo,armi2,aw,cameron,fou1,fou2,gould}).
}
\end{prob}

\begin{prob}
{\rm
The notion of conjugation $\sim_p$  defined in (\ref{econ2}) is very important in symbolic dynamics in connection
with the Williams Conjecture \cite{will}. Characterize  $\sim_p$ in $T(X)$, $P(X)$ and $\mi(X)$ for an infinite set $X$.
(Kudryavtseva and Mazorchuk \cite{KuMa07} have characterized $\sim_p^*$ (the transitive closure of $\sim_p$) in
$T(X)$, $P(X)$ and $\mi(X)$ for a finite $X$, and in $\mi(X)$ for a countably infinite $X$.)
}
\end{prob}

\section{Acknowledgments}
The authors would like to thank the referee for the very careful reading of our paper and excellent suggestions,
especially the one regarding Definition~4.20. We are grateful to Michael Kinyon whose questions and problems on conjugation in semigroups led to this paper.

\end{document}